\documentclass{amsart}
\usepackage[letterpaper, margin=1in]{geometry}
\usepackage[utf8]{inputenc}
\usepackage{amssymb,latexsym,amsthm,amsmath,mathtools}
\usepackage{graphicx}
\usepackage[alphabetic]{amsrefs}
\usepackage{tikz}
\usetikzlibrary{patterns}
\usepackage{verbatim}
\usepackage{pinlabel}
\usepackage{verbatim}
\usepackage{bm}
\usepackage{xspace}
\usepackage{url}
\usepackage{import}
\usepackage{enumitem} % So you can reference enumerated items
\ifx\pdfoutput\undefined
\usepackage{hyperref}
\else
\usepackage[pdftex,colorlinks=true,linkcolor=blue,urlcolor=blue,citecolor=violet]{hyperref}
\fi
\usepackage[nameinlink,capitalise,noabbrev]{cleveref}

\newcommand{\R}{\mathbb{R}}
\newcommand{\C}{\mathbb{C}}
\newcommand{\N}{\mathbb{N}}
\newcommand{\SL}{\mathrm{SL}}

\newcommand{\SO}{\mathrm{SO}}
\newcommand{\Q}{\mathbb{Q}}
\newcommand{\Z}{\mathbb{Z}}
\newcommand{\GL}{\mathrm{GL}}

\renewcommand{\emptyset}{\varnothing}

\DeclareMathOperator{\Aut}{Aut}
\DeclareMathOperator{\Trans}{Trans}

\DeclareMathOperator{\Isom}{Isom}
\DeclareMathOperator{\Aff}{Aff}
\DeclareMathOperator{\Homeo}{Homeo}
\DeclareMathOperator{\Diffeo}{Diffeo}
\DeclareMathOperator{\Map}{Map}

\DeclareMathOperator{\Id}{Id}
\DeclareMathOperator{\diam}{diam}

\theoremstyle{plain}
\newtheorem{theorem}{Theorem}
\numberwithin{theorem}{section}
\newtheorem{corollary}[theorem]{Corollary}
\newtheorem{proposition}[theorem]{Proposition}

\newtheorem{lemma}[theorem]{Lemma}

\theoremstyle{definition}
\newtheorem{definition}[theorem]{Definition}

\theoremstyle{remark}
\newtheorem{remark}[theorem]{Remark}
\newtheorem{question}[theorem]{Question}

\newtheorem{disclaimer}[theorem]{Disclaimer}

\begin{document}

	\title[Realizing groups as symmetries of infinite translation surfaces]{Realizing groups as symmetries of infinite translation surfaces}
	
	\author[Artigiani]{Mauro Artigiani}
	\address{Universidad Nacional de Colombia - Bogot\'a}
	\email{martigiani@unal.edu.co}
	\urladdr{\url{https://m-artigiani.github.io/}}
	
	\author[Randecker]{Anja Randecker}
	\address{Heidelberg University, Germany}
	\email{arandecker@mathi.uni-heidelberg.de}
	\urladdr{\url{https://www.mathi.uni-heidelberg.de/~arandecker/}}
	
	\author[Sadanand]{Chandrika Sadanand}
	\address{Bowdoin College, MSE, USA}
	\email{c.sadanand@bowdoin.edu}
	\urladdr{\url{https://sites.google.com/view/chandrika-sadanand}}
	
	\author[Valdez]{Ferr\'an Valdez}
	\address{Centro de Ciencias Matem\'aticas, UNAM Campus Morelia, M\'exico}
	\email{ferran@matmor.unam.mx}
	\urladdr{\url{https://www.matmor.unam.mx/~ferran/}}
	
	\author[Weitze-Schmith\"{u}sen]{Gabriela Weitze-Schmith\"{u}sen}
	\address{Universit\"{a}t des Saarlandes, Germany}
	\email{weitze@math.uni-sb.de}
	\urladdr{\url{https://www.math.uni-sb.de/ag/weitze/index.php/en/people/weitze-schmithuesen}}
	
	\date{February 20, 2026}
	
	
	\begin{abstract}
		We provide a complete classification of groups that
		can be realized as isometry groups of a translation surface $M$  with non-finitely generated fundamental group
		and no planar ends. Furthermore, we demonstrate that if $S$ has no
		non-displaceable subsurfaces and its space of ends is self-similar, then every
		countable subgroup of $\GL^+(2,\R)$ can be realized as the Veech group of a
		translation surface $M$ homeomorphic to $S$. The latter result generalizes and
		improves upon the previous findings of
		Przytycki--Valdez--Weitze-Schmith\"{u}sen~\cite{PSV} and
		Maluendas--Valdez~\cite{Maluendas-Valdez17}. To prove these results, we adapt ideas from the work of
		Aougab--Patel--Vlamis~\cite{APV}, which focused on hyperbolic surfaces, to translation surfaces.
	\end{abstract}

	\maketitle

	\section{Introduction}
	\label{sec:introduction}
	A topological surface is of \emph{infinite type} if its fundamental group is not
	finitely generated, some examples can be seen in \cref{fig:examples_inf_surfs}.
	Infinite-type surfaces and their mapping class groups, called \emph{big mapping
		class groups}, have been intensively studied in the last few years and the
	literature on the subject is already quite vast. We refer to the
	survey~\cite{AramayonaVlamis} for an overview of the subject.
	
	A translation structure on a topological surface is an atlas of charts
	whose transition functions, except in the neighborhood of singular points, are
	translations. The study of \emph{compact} translation surfaces is a classical
	and well-developed area of research, with many profound connections to diverse
	branches of mathematics, see, for example,~\cites{Wright15, Zorich} for an
	introduction. In the last 15 years there has also been an interest in
	\emph{infinite-type} translation surfaces, their geometry and dynamics, see the
	upcoming book~\cite{DHV1}.
	
	Symmetries on a translation surface are captured by two groups: the isometry
	group, with respect to their natural flat metric, and the Veech
	group.
	In this paper, we study whether and how the topological type of an infinite-type
	surface imposes conditions on which (abstract) groups can be realized as
	isometry and Veech groups of translation surfaces on a topological surface of
	the given (infinite) type. More precisely, we consider the following.
	
	\begin{question}
		Given a topological surface $S$ of infinite type and a countable
		group $G$, can we realize $G$ as the isometry group of a translation
		structure $M$ on $S$? Given a countable subgroup $G$ of $\GL^+(2,\R)$, can $G$ be realized as the Veech group of a translation structure on $S$?
	\end{question}
	
	Our results concerning the isometry groups are analogous to, and inspired by,
	those recently obtained by Aougab, Patel, and Vlamis in~\cite{APV} for complete
	hyperbolic metrics on surfaces of infinite type. Concretely, we show that the
	topology of the surface imposes conditions, which we classify completely,
	on which groups can be realized as the isometry group of a translation surface.
	
	On the other hand, we provide evidence to support the conjecture that the
	topology of the surface does not impose conditions on which (countable)
	subgroups of $\GL^+(2,\mathbb{R})$ can be Veech groups. The results for the Veech groups
	generalize a series of previous statements by some of the authors and their
	collaborators, see \cref{sec:previousresultsVeech} for more details.
	
	\subsection{Statement of the results}
	
	Ker\'ekj\'art\'o and Richards showed that an orientable surface of infinite type
	is classified using three invariants: the genus $g(S)\in\Z_{\geq 0}\cup\{\infty\}$, the space of ends $E(S)$, and the
	subspace of ends that are accumulated by genus $E^g(S) \subset
	E(S)$, see~\cite{Richards63}.
	
	We organize our investigation and results by partitioning the set of infinite-type surfaces into a trichotomy. Many of our theorems have a triple of statements, one for each part. Roughly
	speaking, the trichotomy generalizes the cases where there is only one,  exactly two, or more than three ends, we refer to \cref{fig:examples_inf_surfs} for examples of
	these three cases and \cref{sec:topo-preliminaries} for the definitions. The partition is created using the terminology of Mann and Rafi~\cite{Mann-Rafi_23} by distinguishing
	infinite-type surfaces whose end space is self-similar and surfaces which have a
	non-displaceable subsurface of finite type. Moreover, we will use the topological notion of
	\emph{translatable surfaces} introduced by Schaffer-Cohen
	in~\cite{schaffer-cohen_20}.
	We stress that the translatability of a surface is a topological property that is unrelated to the geometric concept of a translation structure on a surface.
		
	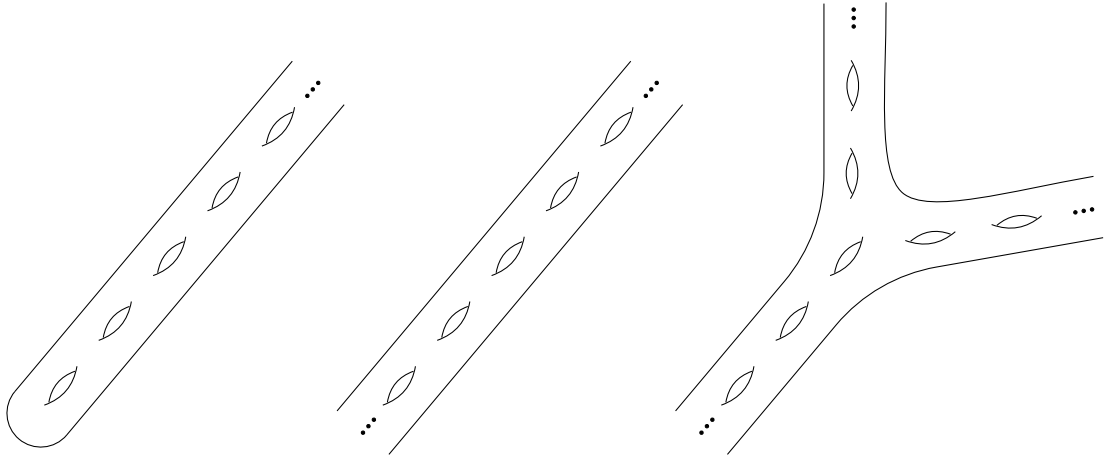
\begin{figure}[tb]
	\centering
	\begin{tikzpicture}[scale=.9]
		
		\begin{scope}[rotate=50, scale=0.25]
			\draw (-2,-2) -- (23.5,-2);
			\draw (-2,2) -- (23.5,2);
			\draw (-2,2) arc (90:270:2cm);
			
			\foreach \g in {0,5,...,20}{
				\begin{scope}[xshift=\g cm]
					\draw (-1.2,0.05) to[bend left] (1.2,0.05);
					\draw (-1.5,0.15) to[bend right] (1.5,0.15);
				\end{scope}
			}
			
			\draw[fill] (22.5,0) circle (0.1cm);
			\draw[fill] (23,0) circle (0.1cm);
			\draw[fill] (23.5,0) circle (0.1cm);
		\end{scope}
		
		\begin{scope}[xshift=5cm, rotate=50, scale=0.25]
			\draw (-3.5,-2) -- (23.5,-2);
			\draw (-3.5,2) -- (23.5,2);
			
			\foreach \g in {0,5,...,20}{
				\begin{scope}[xshift=\g cm]
					\draw (-1.2,0.05) to[bend left] (1.2,0.05);
					\draw (-1.5,0.15) to[bend right] (1.5,0.15);
				\end{scope}
			}
			
			\draw[fill] (-2.5,0) circle (0.1cm);
			\draw[fill] (-3,0) circle (0.1cm);
			\draw[fill] (-3.5,0) circle (0.1cm);
			\draw[fill] (22.5,0) circle (0.1cm);
			\draw[fill] (23,0) circle (0.1cm);
			\draw[fill] (23.5,0) circle (0.1cm);
		\end{scope}
		
		\begin{scope}[xshift=10cm, rotate=50, scale=0.25]
			\draw (-3.5,-2) -- (6.5,-2) arc (90:50:10cm) -- ++ (-40:10cm);
			\draw (-3.5,2) -- (6.5,2) arc (-90:-50:10cm) -- ++ (40:10cm);
			\draw (23,8) .. controls +(-140:5cm) and +(90:2cm) ..
			(15,0) .. controls +(-90:2cm) and +(140:5cm) .. (23,-8);
			
			\foreach \g in {0,5,10}{
				\begin{scope}[xshift=\g cm]
					\draw (-1.2,0.05) to[bend left] (1.2,0.05);
					\draw (-1.5,0.15) to[bend right] (1.5,0.15);
				\end{scope}
			}
			\begin{scope}[xshift=14 cm, yshift=3cm, rotate=40]
				\draw (-1.2,0.05) to[bend left] (1.2,0.05);
				\draw (-1.5,0.15) to[bend right] (1.5,0.15);
			\end{scope}
			\begin{scope}[xshift=18 cm, yshift=6.3cm, rotate=40]
				\draw (-1.2,0.05) to[bend left] (1.2,0.05);
				\draw (-1.5,0.15) to[bend right] (1.5,0.15);
				\draw[fill] (3.5,0) circle (0.1cm);
				\draw[fill] (4,0) circle (0.1cm);
				\draw[fill] (4.5,0) circle (0.1cm);
			\end{scope}
			\begin{scope}[xshift=14 cm, yshift=-3cm, rotate=-40]
				\draw (-1.2,0.05) to[bend left] (1.2,0.05);
				\draw (-1.5,0.15) to[bend right] (1.5,0.15);
			\end{scope}
			\begin{scope}[xshift=18 cm, yshift=-6.3cm, rotate=-40]
				\draw (-1.2,0.05) to[bend left] (1.2,0.05);
				\draw (-1.5,0.15) to[bend right] (1.5,0.15);
				\draw[fill] (3.5,0) circle (0.1cm);
				\draw[fill] (4,0) circle (0.1cm);
				\draw[fill] (4.5,0) circle (0.1cm);
			\end{scope}
			
			\draw[fill] (-2.5,0) circle (0.1cm);
			\draw[fill] (-3,0) circle (0.1cm);
			\draw[fill] (-3.5,0) circle (0.1cm);
		\end{scope}
	\end{tikzpicture}
	
	\caption{Examples of infinite-type surfaces. Right to left: a surface with self-similar end space, a translatable surface, and a surface with a non-displaceable subsurface of finite type.}
	\label{fig:examples_inf_surfs}
\end{figure}

	We write $\Isom(M)$ for the isometry group of a translation surface $M$. Our main result is the following.
	
	\begin{theorem}\label{thm:main}
		Let $S$ be an infinite-type surface such that $E(S) = E^g(S)$ and $G$ be an
		arbitrary group.
		\begin{enumerate}
			\item Assume that $S$ has self-similar end space. Then there exists a translation structure
			$M$ on $S$ such that $\Isom(M)$ is isomorphic to $G$ if and only if $G$ is \emph{countable}.
			\item Assume that $S$ is translatable and $E(S)$ is not self-similar. Then there exists a translation
			structure~$M$ on $S$ such that $\Isom(M)$ is isomorphic to $G$ if and only if $G$ is \emph{virtually
				cyclic}.
			\item In all other cases, there exists a translation structure $M$
			on $S$ such that $\Isom(M)$ is
			isomorphic to~$G$ if and only if $G$ is \emph{finite}.
		\end{enumerate}
	\end{theorem}
	
	This result is an analogue in the realm of translation surfaces of~\cite{APV}*{Theorem~6.2}.
	We stress however that~\cref{thm:main} is stronger than its analogue for the hyperbolic
	case, since, by using the more restrictive notion of
	translatable surfaces instead of surfaces with doubly-pointed end space\,\footnote{
		A space of ends $E(S)$ is \emph{doubly pointed} if there are exactly two points in it whose $\Homeo(S)$-orbit is finite, see~\cref{def:doubly pointed}.}, we are
	able to obtain a trichotomy. Observe that, if a surface is translatable, it
	cannot contain non-displaceable subsurfaces of finite type.  Using the same three
	categories as in \cref{thm:main}, we can also improve~\cite{APV}*{Theorem~6.2}, in
	the complete hyperbolic case, to a trichotomy, more precisely:
	
	\begin{theorem}
		\label{thm:improvingAPV1}
		Let $S$ be an infinite-type surface such that $E(S) = E^g(S)$ and $G$ be an arbitrary group.
		\begin{enumerate}
			\item Assume that $S$ has self-similar end space. Then there exists a complete hyperbolic metric $M$ on $S$ such that $\Isom(M)$ is isomorphic to $G$ if and only if $G$ is \emph{countable}.
			\item Assume that $S$ is translatable and $E(S)$ is not self-similar. Then there exists a complete hyperbolic metric $M$ on $S$ such that $\Isom(M)$ is isomorphic to $G$ if and only if $G$ is \emph{virtually cyclic}.
			\item In all other cases, there exists a complete hyperbolic metric $M$ on $S$ such that $\Isom(M)$ is isomorphic to $G$ if and only if $G$ is \emph{finite}.
		\end{enumerate}
	\end{theorem}
	
	We present the proof of~\cref{thm:improvingAPV1} in~\cref{ssec:proof-APV-improved}.

	\begin{remark}
		The statement in \cite{APV}*{Theorem 6.2} differs from \cref{thm:improvingAPV1} in two parts. First, part~(2) in \cite{APV}*{Theorem 6.2} says that if $E(S)$ is doubly-pointed then the isometry group of any complete hyperbolic metric of $S$ is virtually cyclic. Second, (3) in the same statement says that if $S$ contains a compact non-displaceable surface then there exists a complete hyperbolic metric if and only if $G$ is finite. Moreover, as remarked by Aougab--Patel--Vlamis, \cite{APV}*{Theorem 6.2} is not a trichotomy.
	\end{remark}
	
	We obtain our main theorem as the corollary of several results, each dealing
	with one case and implication of the trichotomy. From these results, it is clear that
	$E^g(S) = E(S)$ is not an optimal hypothesis. We still decided to present our main result as stated in~\cref{thm:main} for the sake of readability.
	
	We begin by analyzing the self-similar case. In this case, we prove the following.
	
	\begin{theorem}
		\label{thm:realization-isometries-selfsimilar}
		Let $S$ be an infinite-type surface such that $g(S)>0$, $E(S)$ is
		self-similar, and where every subsurface $S'\subset S$ of finite type is displaceable.
		Let $G$ be any \emph{infinite} group. Then $G$ is \emph{countable} if and
		only if there exists a translation structure $M$ on $S$ such that
		$\Isom(M)$ is isomorphic to~$G$.
	\end{theorem}
	
	We remark that the surfaces considered in~\cite{APV}*{Theorem~6.2} are a strict
	subclass of those considered in the result above. In fact, there are infinitely
	many examples of surfaces satisfying the hypothesis of
	\cref{thm:realization-isometries-selfsimilar} having only one end
	accumulated by genus and infinitely many planar ends. For example, consider the
	realization of the Loch Ness monster\,\footnote{The \emph{Loch Ness monster} is
		the unique surface of infinite type with only one end and which is accumulated
		by genus.} as tori along a lattice in $\mathbb{R}^2$, and then remove from each
	torus a copy of $\omega^n+1$, where $\omega$ is the first infinite ordinal and $n\in\N$ is fixed. Conversely, every
	infinite-genus surface with self-similar end space has no finite-type
	non-displaceable subsurfaces. Indeed, if $S$ is a surface of infinite genus with
	self-similar end space, \cite{Mann-Rafi_23}*{Proposition~3.1} implies
	that the mapping class group of $S$, $\Map(S)$, is coarsely bounded. However,
	\cite{Mann-Rafi_23}*{Theorem~1.9} yields that the existence of a finite-type
	non-displaceable subsurface implies that $\Map(S)$ is not coarsely bounded.
	
	For surfaces which additionally do not have planar ends,
	\cref{thm:realization-isometries-selfsimilar} can be deduced from~\cite{APV}*{Theorem~6.2}, see \cref{rem:deduction-proof}.
	To cover all remaining cases (and for readers who are not familiar with~\cite{APV}),
	we present a constructive proof of \cref{thm:realization-isometries-selfsimilar},
	inspired by~\cite{APV}*{Theorem~6.2}.
	Furthermore, this constructive approach is used to obtain the results for Veech groups (see \cref{sec:proofs-Veech}), where one cannot directly build upon \cite{APV}.	
	
	In the case of doubly-pointed end space, which includes the case of translatable
	surface with non-self-similar end space, we prove that the isometry group is always virtually cyclic.
	
	\begin{theorem}
		\label{thm:realization-isometries-doubly-pointed}
		Let $S$ be an infinite-type surface. If the end space of $S$ is doubly-pointed, then
		for any translation structure $M$ on $S$, the group $\Isom(M)$ is virtually
		cyclic.
	\end{theorem}
	
	Before we continue, let us comment on the difference between the previous two
	results. The proof of~\cref{thm:realization-isometries-selfsimilar} relies on
	the notion of \emph{radially symmetric} end spaces, which was introduced
	in~\cite{APV} and is equivalent to self-similarity. This definition provides
	what can be considered a ``normal form'' for the end space of $S$, allowing for
	an if-and-only-if statement for countable groups. The lack of such normal form
	for surfaces with doubly-pointed end space is the reason we are only able to
	obtain one direction in~\cref{thm:realization-isometries-doubly-pointed} and
	it is one of the reasons that led us to consider the narrower class of
	translatable surfaces. For translatable surfaces, we obtain the following.
	
	\begin{theorem} \label{thm:realization-isometries-translatable}
		Let $S$ be an infinite-type surface with doubly-pointed end space, $g(S)>0$, and $G$ be
		a group. Then the following are
		equivalent.
		\begin{enumerate}
			\item $S$ is translatable.
			\item There exists a translation structure $M$ on $S$ with $\Isom(M) \cong G$ if and only if $G$ is virtually cyclic.
			\item There exists a translation structure $M$ on $S$ such that $\Isom(M)$ contains an element of infinite order.
		\end{enumerate}
	\end{theorem}
	
	Finally, in the case when $S$ has a non-displaceable subsurface of finite type,
	we prove that the isometry group itself is necessarily finite.
	
	\begin{theorem}
		\label{thm:non-displaceable-implies-finite-isometries}
		Let $S$ be an infinite-type surface with a non-displaceable subsurface of
		finite type. Then, for any translation structure $M$ on $S$, we have that
		$\Isom(M)$ is finite.
	\end{theorem}
	
	Assuming that all ends are accumulated by genus, we can realize any finite
	group as an isometry group, regardless of the underlying topology.
	
	\begin{theorem}
		\label{thm:realization-finite-groups}
		Let $S$ be an infinite-type surface such that $E(S)=E^g(S)$ and $G$ be any
		finite group. Then there exists a translation structure $M$ on $S$ such that
		$\Isom(M)$ is isomorphic to~$G$.
	\end{theorem}
	
	As with \cref{thm:realization-isometries-selfsimilar}, there are
	two possible approaches for the proof of \cref{thm:realization-finite-groups}.
	We focus in \cref{ssec:proof-realization-finite-groups} on presenting a
	constructive proof because its content can be extended to obtain similar results
	for Veech groups.
	
	Before switching to Veech groups, let us consider the special case when the end
	space is \emph{countable}. Then there exists a countable ordinal $\alpha\geq 0$ and a positive integer
  $n$ such that the end space $E(S)$ is  determined, up to homeomorphism, by the ordered pair $(\alpha,n)$ which is known as the \emph{characteristic system} of $E(S)$,	see~\cref{sec:topo-preliminaries} for more details. In this case, we
	obtain the following result, which is the translation surface analogue
	of~\cite{APV}*{Theorem~4.18}.
	
	\begin{theorem}
		\label{thm:countable-case}
		Let $S$ be an infinite-type surface such that $E(S)=E^g(S)$ is countable with
		characteristic system $(\alpha,n)$, and let $G$ be an arbitrary group.
		\begin{enumerate}
			\item If $n=1$ there exists a translation structure $M$ on $S$ such that
			$\Isom(M)$ is isomorphic to $G$ if and only if $G$ is \emph{countable}.
			\item If $n=2$ and $\alpha$ is a \emph{successor} ordinal or zero, there exists a
			translation structure $M$ on $S$ such that $\Isom(M)$ is isomorphic to $G$ if and only if $G$ is \emph{virtually cyclic}.
			\item If $n\geq 3$, or $n=2$ and $\alpha$ is a \emph{limit} ordinal, there
			exists a translation structure $M$ on $S$ such that $\Isom(M)$ is isomorphic to $G$ if and only if $G$ is \emph{finite}.
		\end{enumerate}
	\end{theorem}
	
	\begin{remark}
		Let $M$ be a translation surface. We denote by $\Trans(M) \le \Isom(M)$ the
		group of translations of $M$, which are isometries that, in local coordinates,
		are translations. Our proofs  of \crefrange{thm:main}{thm:countable-case} are
		constructive. Notably, each translation structure $M$ we construct in our proofs
		not only realizes the abstract group in question as its isometry group, but also
		possesses the property $\Isom(M)=\Trans(M)$. This fits in line with a recent result
		by Hidalgo and Morales in~\cite{HidalgoMorales}, who showed that every countable
		group can be realized as the group of automorphisms of some origami\,\footnote{Also
			called square-tiled surface.}, which, if the group is infinite, is homeomorphic
		to the Loch Ness monster. Note that for~\cite{HidalgoMorales}, an automorphism is a translation which respects the tessellation by squares of the origami.
	\end{remark}
		
	We now turn our attention to Veech groups. On a translation surface $M$, we
	denote by $\Aff(M)$ the group of $\R$--affine orientation-preserving
	homeomorphisms of $M$. Since the transition functions are translations, any
	such affine homeomorphism has a constant derivative, which is an element of
	$\GL^+(2,\R)$. By definition, the \emph{Veech group} $\Gamma(M)$ of $M$ is the group of
	derivatives of affine homeomorphisms. The relation between the group of
	translations and the Veech group can be
	summarized by the following short exact sequence
	\[
	1 \to \Trans(M) \to \Aff(M) \to \Gamma(M) \to 1.
	\]
	Moreover, we recall that $\Trans(M) \le \Isom(M)$, see \cref{sec:transsurf-preliminaries} for
	more details.
	
	The problem of which groups can be Veech groups is still open even for compact
	translation surfaces, see \cref{sec:SymmetriesCompact}, and seems to be harder
	than the isometry groups. Part of the difficulty in studying Veech groups comes
	from the fact that these groups need not act properly discontinuously
	on the surface. Our first result for Veech groups is similar to
	\cref{thm:realization-isometries-selfsimilar}, discussing which groups can
	appear for surfaces which have self-similar end space.
	
	\begin{theorem}
		\label{thm:realization-Veech-selfsimilar}
		Let $S$ be an infinite-type surface such that $g(S)>0$, $E(S)$ is
		self-similar, and every subsurface $S'\subset S$ of finite type is displaceable. Let
		$G$ be any \emph{countable} subgroup of $\GL^+(2,\mathbb{R})$. Then there exists a
		translation structure $M$ on $S$ such that $\Aff(M)\to \Gamma(M)$ has trivial
		kernel and $\Gamma(M)$ is isomorphic to $G$.
	\end{theorem}
	
	Moreover, we show that finite groups can always be realized, independent of
	the topology of $S$, similar to \cref{thm:realization-finite-groups}.
	
	\begin{theorem}
		\label{thm:realization-finite-Veech-groups}
		Let $S$ be an infinite-type surface such that $E(S)=E^g(S)$ and let $G<\GL^+(2,\R)$
		be a \emph{finite} group. Then there exists a translation structure $M$ on $S$ such
		that $\Aff(M)\to \Gamma(M)$ has trivial kernel and~$\Gamma(M)$ is isomorphic to
		$G$.
	\end{theorem}
	
	We also obtain similar results for virtually cyclic groups for positive-genus translatable surfaces, more precisely: 
	
	\begin{theorem}
		\label{thm:realization-vistually-cyclic-Veech-groups}
		Let $S$ be a translatable surface such that $g(S)>0$ and $E(S)$ is not self-similar. Let $G<\GL^+(2,\R)$ be any virtually cyclic group. Then there exists a translation structure $M$ on $S$ such that  $\Aff(M)\to \Gamma(M)$ has trivial kernel and~$\Gamma(M)$ is isomorphic to $G$.
	\end{theorem}

	Another indication that Veech groups are harder to classify is that they do not
	need to be countable. In fact, we prove the following extension
	of~\cite{Maluendas-Valdez17}*{Theorem~0.4} where the authors only considered
	surfaces without planar ends:
	\begin{theorem}
		\label{thm:uncountable-veech-no-restrictions}
		Let $S$ be an infinite-type surface and
		\begin{equation}
			\label{eq:group-P}
			P \coloneqq \left\{\begin{pmatrix}1 & s\\0 & t\end{pmatrix}: s\in\R
			, t>0\right\} .
		\end{equation}
		Then there exists a translation structure $M$ on $S$ such that $\Gamma(M)=P$.
	\end{theorem}
	
	\subsection{Symmetries of compact translation surfaces}
	\label[section]{sec:SymmetriesCompact}
	
	A natural question is why we restrict ourselves to surfaces of infinite type. In
	order to make this clear, let us discuss for a moment our guiding question in
	the case of \emph{finite}-type translation surfaces. Since translation surfaces
	are Riemann surfaces, see~\cref{rmk:defs-TS}, the isometry group is a subgroup
	of the group of automorphisms of the Riemann surface. The Hurwitz automorphism
	theorem~\cite{Hurwitz}*{Abschnitt II.7} then implies that the isometry group is
	always finite, with order bounded by $84(g-1)$, where $g\ge 2$ is the genus of
	the surface, see, for example,~\cite{FarbMargalit}*{Theorem~7.4}.
	It is worth mentioning that any finite group occurs as automorphism
	group of a compact translation surface, more precisely of an origami
	(see~\cite{Hidalgo})\footnote{Here, automorphisms are meant as automorphisms of the translation structure again.}.
	However, Schlage-Puchta and Weitze-Schmith\"{u}sen show that the
	group of translations  of a compact translation surface of genus $g$ has at most
	$4g-4$ elements and has exactly $4g-4$ elements if and only if it is a normal origami
	in the principal stratum, see~\cite{SP-WS17} for more details. In particular, the list of finite groups
	which can be realized as automorphism groups of translation surfaces of a fixed genus is exclusive.
	
	Regarding Veech groups, we know that they are always non-cocompact Fuchsian groups,
	see~\cite{HubertSchmidt:Veech}. Typically, they are quite small, and in fact, a
	generic\,\footnote{``Generic'' is here meant with respect to the Masur--Veech
		measure.} surface has trivial Veech group~\cite{Moller:affine}. However, there
	exist exceptional surfaces, called Veech (or lattice) surfaces, whose Veech
	group is a lattice. In this case, the geodesic flow (also called directional
	flow) exhibit special dynamical properties~\cites{Veech,HubertSchmidt:Veech}.
	Moreover, there are examples of Veech groups that are infinitely
	generated~\cite{HubertSchmidt:infinite}.
	
	The question of which Fuchsian groups can be realized as Veech groups is still open and is a
	difficult one. For instance, it is not known whether there exists a Veech group
	which is generated by a single hyperbolic matrix.\footnote{This is called the
		``lonely guy conjecture''.} Moreover, it is not known if there is a
	non-elementary Veech group of the $2^\text{nd}$ kind. For more information on
	Veech groups of compact translation surfaces, see~\cite{Lehnert17}.
	
	\subsection{Previous results on Veech groups of infinite surfaces}
	\label{sec:previousresultsVeech}
	Our results for Veech groups extend previous results from some of the authors
	and their collaborators: Przytycki, Valdez, and Weitze-Schmith\"usen showed
	in~\cite{PSV} that every countable group $G < \GL^+(2,\R)$ can be realized as
	the Veech group of some translation structure on the Loch Ness monster.
	Maluendas and Valdez later expanded this result in~\cite{Maluendas-Valdez17} as
	follows.
	
	Assume that $S$ is an infinite-type surface such that $E(S) = E^g(S)$. If
	\begin{itemize}
		\item $S$ is the blooming Cantor tree, whose end space is a Cantor set,
		see \cref{fig:bloomingcantor}, or
		\item $E(S)$ is homeomorphic to the union of a a Cantor set and a discrete set of points
		accumulating to one point of the Cantor set, or
		\item $E(S)$ is homeomorphic to $\omega^k + 1$ for $k\ge 1$,
	\end{itemize}
	then, as in~\cite{PSV}, one can realize any countable group as the Veech group
	of some translation structure on~$S$. We remark that all these surfaces have
	self-similar end space, and therefore our
	\cref{thm:realization-Veech-selfsimilar} generalizes these results.
	
	\begin{figure}[tb]
		\centering
		\includegraphics[height=.2\textheight]{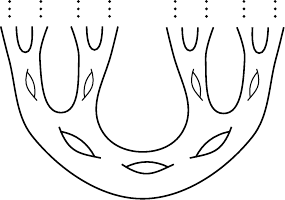}
		\caption{The blooming Cantor tree.}
		\label{fig:bloomingcantor}
	\end{figure}
	
	When considering surfaces whose space of ends is not self-similar, less is known about realizability for Veech groups so far. As one of the few results, Morales and Valdez
	showed the following, using the Hooper--Thurston--Veech construction~\cite{Hooper}.
	
	\begin{theorem}[Morales, Valdez~\cite{Morales_Valdez_20}]
		\label{teo:HTV-construction}
		Let $S$ be an infinite-type surface. Then there exist uncountably many
		half-translation structures $M$ whose Veech group $\Gamma(M)$ contains a free
		group with two generators.
	\end{theorem}
	
	We recall that a \emph{half-translation} structure is similar to a translation structure, but with transition maps (except at singular points) being translations, possibly composed with a rotation by $\pi$.
	
	\subsection{Applications to big mapping class groups} In
	\cref{lemma:isom-into-map} we show that, for any translation structure
	$M$ on a topological surface $S$, the natural map $\Isom(M)\to\Map(S)$, which
	associates to each isometry its mapping class, is an injective homomorphism. As
	a consequence, \cref{thm:realization-isometries-selfsimilar} can be applied to
	extend Corollary 1.4 in~\cite{APV}. More precisely:
	
	\begin{corollary}
		\label{corollary:applications-to-bigmcgs}
		Let $S$ be an infinite-type surface such that $g(S)>0$, $E(S)$ is
		self-similar, and where every subsurface $S'\subset S$ of finite type is
		displaceable. Then $\Diffeo(S)$, $\Homeo(S)$ and $\Map(S)$ are not linear,
		are not (cyclically or linearly) orderable and are not coherent.
	\end{corollary}

	\subsection{Methods}
	
	Our primary inspiration is the work of Aougab--Patel--Vlamis in~\cite{APV}.
	Following their approach, we exploit some variation of a Cayley graph of a given
	group $G$ to guide our construction of a translation structure. The core
	idea is to construct a vertex surface ensuring that it has the appropriate end
	space. To do this, we follow and adapt earlier techniques of
	Allcock~\cite{Allcock06}, Przytycki--Valdez--Weitze-Schmith\"usen~\cite{PSV},
	Maluendas--Valdez~\cite{Maluendas-Valdez17} and Maluendas--Randecker~\cite{Randecker2016}. However, a major difference is that flat surfaces
	lack the hyperbolic collar lemma, which is a crucial tool for~\cite{APV}. To
	circumvent this difficulty, we simplify their construction by eliminating the
	edge surfaces and gluing using~slits.
	
	The use of the Cayley graph of $G$ guarantees that the group naturally acts
	isometrically on the constructed surface. Therefore, to complete our
	proofs, we must show that there are no other isometries.
	
	We emphasize that all our proofs are \emph{constructive}, providing a step-by-step
	procedure to construct a translation surface $M$ from a topological surface $S$
	and an abstract group $G$. For the sake of clarity, we provide a blueprint for
	each construction, aiming to streamline the presentation as much as possible.
	Finally, since some of the results from~\cite{APV} are independent of the
	hyperbolic metric, they apply to our setting directly. We do not reproduce these proofs,
	preferring to refer the reader to the well-written original~\cite{APV}.
	
	\subsection{Open questions}
	The constructions outlined earlier have two drawbacks: the resulting surface $M$
	always has infinite area, and the actions of both $\Isom(M)$ and $\Aff(M)$ have no
	fixed points. While the absence of fixed points is desirable for $\Isom(M)$,
	which should act properly discontinuously, it is not for~$\Aff(M)$, and the
	infinite area of $M$ is unnatural.  This motivates the following questions:
	
	\begin{question}
		\label{q:finite-area-isom-veech}
		Let $S$ be an infinite-type surface and $G$ an (abstract) group. Does there exist a translation structure~$M$ on $S$ of \emph{finite area} such that $G$ is isomorphic to $\Isom(M)$?
	\end{question}
	
	Even for simple instances, this question remains open. For example, to the best
	of our knowledge, it is still unknown if when $E(S)=E^g(S)$ has two elements,
	there exists a translation structure $M$ on $S$ such that $\Gamma(M)=\SL(2,\Z)$.
	
	However, in upcoming work the authors show that there exist translation structures on the Loch Ness monster that can realize any free group of any rank as Veech group.

	\subsection{Organization of the paper}
	The paper is organized as follows. We begin in~\cref{sec:topo-preliminaries} by
	recalling basic topological results such as the classification of infinite-type
	surfaces and the Cantor--Bendixson derivative. Then we introduce and discuss the
	notions of self-similar and doubly-pointed end spaces. Of
	particular importance is \cref{ssec:translatable-surfaces} where we
	define translatable surfaces and present some useful results regarding these. In
	\cref{ssec:useful-results} we present two important results for groups of
	homeomorphisms acting properly discontinuously on a surface.
	In \cref{sec:transsurf-preliminaries}, we give the definition of translation
	surfaces and of the groups we will be interested in. We also recall the
	end-grafting construction, which is a fundamental tool in what follows. Finally,
	we prove our results about isometry groups in~\cref{sec:proofs-isometries} and
	the ones about Veech groups in~\cref{sec:proofs-Veech}.
		
	\section*{Acknowledgements}
	This research project was supported by AIM’s research
	program Structured Quartet Research Ensembles (SQuaREs).
	A.\,R.\ acknowledges funding from Deutsche Forschungsgemeinschaft (DFG,
	German Research Foundation) under RA 2857/3-1 in the Priority Programme
	``Geometry at Infinity'' (SPP 2026) and Germany’s Excellence Strategy - EXC
	2181/1 – 390900948 (the Heidelberg STRUCTURES Excellence Cluster).
	F.\,V.\ would like to thank the following grants UNAM PAPIIT IN-101422 and IN106925.
	
	\section{Topological preliminaries}
	\label{sec:topo-preliminaries}
	
	In this section, we introduce the topological preliminaries that are necessary
	for the subsequent discussions. Let $S$ be an orientable surface of genus
	$g(S)\in\Z_{\geq 0}\cup{\infty}$ which is not necessarily compact. We are
	mainly concerned with surfaces of infinite type, that is, surfaces for
	which~$\pi_1(S)$ is not finitely generated. We denote by $E(S)$ and
	$E^g(S)\subset E(S)$ the spaces of ends of $S$ and of the ends of $S$
	accumulated by genus, respectively. Note that $E^g(S)\neq\emptyset$ if and only if the surface $S$ has infinite genus. Both of these spaces can be embedded in the Cantor set as
	closed subsets. A detailed discussion of ends of surfaces can
	be found in \cite{Richards63}.
	
	The following classical result by Ker\'ekj\'art\'o and Richards (see
	[\emph{Ibid.}]) provides a topological classification of orientable surfaces:
	
	\begin{theorem}[Topological classification of orientable surfaces]
		\label{thm:classification-surfaces}
		Two orientable surfaces $S$ and $S'$ are homeomorphic if and only if
		$g(S)=g(S')$ and $E^g(S)\subset E(S)$ is homeomorphic to $E^g(S')\subset E(S')$
		as pair of nested topological spaces. Moreover, for every $g\in\Z_{\geq
			0}\cup\{\infty\}$ and pair of closed subsets $E^g\subset E$ of the Cantor set,
		there exists an orientable surface $S$ such that $g(S)=g$ and $E^g(S)\subset
		E(S)$ is homeomorphic to~$E^g\subset E$.
	\end{theorem}
	
	In addition to surfaces, we will also work with end spaces of graphs, which can
	be defined in terms of equivalence classes of infinite paths. Recall that every open subset $U\subset S$ determines an open set $U^*\subset E(S)$. For a general
	discussion on ends of manifolds, we refer to the classical work of
	Raymond~\cite{Raymond60}.
	
	Some of our results concern only surfaces with countable end spaces, so we
	recall notation and facts from descriptive set theory which are related
	to countable compact Hausdorff spaces. For more details, see,
	for example, \cite{Kechris}. Given an ordinal $\alpha$, we identify it with the set
	containing the previous ordinals: $\alpha = \{\beta: \beta<\alpha\}$, with the order topology. We denote
	the first (countably) infinite ordinal $\omega$. An ordinal $\alpha$ is called
	\emph{successor} if $\alpha = \beta + 1$ for some ordinal $\beta$ and
	\emph{limit} if is neither $0$ nor a successor.
	
	If $X$ is a topological space, we
	define
	\[
	X' \coloneqq \{x\in X: x \text{ is an accumulation point of } X\}.
	\]
	The set $X'$ is called the \emph{Cantor--Bendixson derivative} of $X$.
	For an ordinal $\alpha$, we inductively define
	\[
	\begin{split}
		X^0 & \coloneqq X \\
		X^{\alpha+1} & \coloneqq (X^{\alpha})' \\
		X^{\lambda} & \coloneqq \bigcap_{\beta < \lambda} X^{\beta}, \text{ if $\lambda$ is a limit ordinal.}
	\end{split}
	\]
	One can prove that the above sequence stabilizes for every Polish space. The
	smallest ordinal $\beta$ such that~$X^{\beta} = X^{\beta + 1}$ is called the
	\emph{Cantor--Bendixson rank} of $X$. The cardinality of the last non-empty derived set in the above sequence is called the \emph{Cantor-Bendixson degree} of $X$.
	
	The following result due to Mazurkiewicz--Sierpinski~\cite{Mazurkiewicz-Sierpinski} provides a complete classification of countable, compact, Hausdorff spaces:
	
	\begin{theorem}[Topological classification of countable ordinals]
		Let $X$ be a countable, compact, Hausdorff space. Then there exist a countable
		ordinal $\alpha$ and $d\in\N$ such that $X$ is homeomorphic to
		$\omega^{\alpha}\cdot d+1$.
	\end{theorem}
	
	From this result, we get that when $X$ is a countable and compact ordinal, its rank is of the form $\beta=\alpha+1$ for some countable ordinal $\alpha\geq 0$. If $X$ is a countable compact space of Cantor-Bendixson rank $\beta=\alpha+1$ and degree $d$, the ordered pair $(\alpha,d)$ is called the \emph{characteristic system} of $X$.  A finite set of cardinality $d$ has characteristic system $(0,d)$. The compact countable ordinal $\omega^{\alpha}\cdot d+1$ has characteristic system $(\alpha,d)$  if~$\alpha>0$ and  $(\alpha,d+1)$ if $\alpha=0$. For this last case, we interpret $\omega^0=1$ and thus $\omega^0\cdot d +1= d+1$.

	\subsection{Self-similar and doubly pointed end spaces}
	In this section, we
	introduce the key topological notions regarding the end spaces of surfaces that
	are important for this paper. For simplicity, throughout this section,
	$E^g\subset E$ denotes the end space of a surface.
	
	\begin{disclaimer}
		To simplify notation, we endow every subset $X$ of $E$ with the subspace
		topology and consider it as a pair of nested topological spaces $X\cap
		E^g\subset X$. Furthermore, when we say that $X\subset E$ is homeomorphic to
		$X'\subset E$, we mean that $X\cap E^g\subset X$ and $X'\cap E^g\subset X'$ are
		homeomorphic as nested topological spaces.
	\end{disclaimer}
	
	\begin{definition}
		\label{def:self-similar-doubly-pointed}
		The end space $E$ is said to be \emph{self-similar} if for any partition
		$E=E_1\sqcup\dotsb\sqcup E_n$ into pairwise disjoint clopen subsets, there
		exists $i\in\{1,\dotsc,n\}$ and $X\subset E_i$ open such that $X$ is homeomorphic
		to~$E$.
	\end{definition}
	
	\begin{definition}
		\label{def:doubly pointed}
		The end space $E$ is said to be \emph{doubly pointed} if there are exactly two
		points in $E$ whose $\Homeo(S)$--orbit is finite.
	\end{definition}
	
	Every finite self-similar end space has cardinality 1. Similarly, every finite doubly-pointed end space has cardinality 2.
	An infinite countable end space with $E^g=E$ is self-similar (respectively
	doubly pointed) if and only if $E$ has degree 1 (respectively 2). For instance,
	$E=\omega+1$ is self-similar. If $x \coloneqq (\omega+1)'$ (here the exponent denotes the
	Cantor--Bendixson derivative) then $E$ can be written as
	$E\setminus\{x\}=\bigsqcup_{n\in \N}E_n
	$, where each $E_n$ is homeomorphic to $\omega$, $E_n$ is disjoint from the closure
	of $E_m$ whenever $n\neq m$, and $\cap_{n\in\N}\overline{E_n}=x$. Aougab--Patel--Vlamis
	generalized this property to give a normal form for self-similar spaces as
	follows (see~\cite{APV}*{Definition~4.6}).
	
	\begin{definition} \cite{APV}
		\label{def:radially-symmetric}
		A compact Hausdorff topological space has \emph{radial symmetry} if either $E$
		is a singleton or there exists a point $x\in E$ and a collection of pairwise
		homeomorphic, non-compact sets $\{E_n\}_{n\in\N}$ such that $E_n$ is disjoint
		from the closure of $E_m$ whenever $n\neq m$ and such that
		\[
		E\setminus\{x\}=\bigsqcup_{n\in\N}E_n.
		\]
	\end{definition}
	
	\begin{remark}
		\label{rmk:star-point}
		By compactness of $E$ and the topological assumptions on $E_n$, the point $x$
		belongs to the closure of $E_n$ for every $n\in\N$. In this context, $x$ is
		called a \emph{star point} of $E$.
	\end{remark}
	
	The following theorem is a fundamental tool in the proof of our results.
	
	\begin{theorem}\cite{APV}*{Theorem 5.2 }
		\label{thm:self-similar-radially-symmetric}
		Let $E$ be the end space of an orientable surface $S$. Then $E$ is self-similar
		if and only if $E$ has radial symmetry.
	\end{theorem}
	
	\subsection{Translatable surfaces}
	\label{ssec:translatable-surfaces}
	
	The notion of translatable surface allows us to present our main result in the
	form of a trichotomy. In what follows we present their definition, the notion of
	non-displaceable subsurface, and some of the main properties of translatable
	surfaces.
	
	Given a surface $S$, an end $e$ of $S$ and a sequence of curves $\{ \gamma_n\}$ in $S$,
	we say that $\lim_{n \rightarrow \infty} \gamma_n = e$ if and only if for every
	neighborhood $U^*$ of $e$, there exists an $N$ such that for all $n>N$,
	$\gamma_n$ is contained in $U$.

	\begin{definition}[Translatable surfaces] \label{def:translatable} Let $S$ be a
		surface. A homeomorphism $h$ of $S$ is called a \emph{translation} if there
		exist two \emph{distinct} ends $e_-$ and $e_+$ in $E(S)$, such that for any
		closed curve $\alpha$ on $S$, $\lim_{n \rightarrow \infty} h^n(\alpha) = e_+$
		and $\lim_{n \rightarrow \infty} h^{-n}(\alpha) = e_-$.
		
		The surface $S$ is called \emph{translatable} if there exists a translation on $S$.
		
	\end{definition}
	
	\begin{remark}
		The definition of translatable surface appears first
		in~\cite{schaffer-cohen_20} and is slightly different from
		\cref{def:translatable}. There are two main differences. First, in
		\cite{schaffer-cohen_20}, $\lim_{n \rightarrow \infty} \gamma_n = e$ means
		that there exists an~$N$ such that for all $n>N$, $\gamma_n$ is contained in
		$U$ \emph{after some isotopy}. Second, translations are specific mapping
		classes, not homeomorphism. With this consideration, any element in the
		mapping class group of a cylinder (for example, the identity) is a translation given
		that any curve on a cylinder can be isotoped into an end.
		
		For our setup, it fits better to use homeomorphisms instead of mapping
		classes and consider curves instead of isotopy classes of curves. It is
		immediate to check that all translatable surfaces according to
		\cref{def:translatable} are translatable in the sense
		of~\cite{schaffer-cohen_20}.
	\end{remark}

	\begin{figure}[tb]
		\centering
		\includegraphics[width=.8\textwidth]{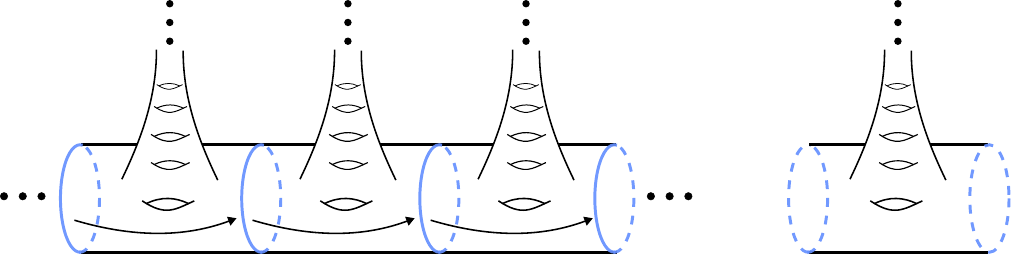}
		\caption{An example of a translatable surface $S$, with the translation $h$
			and some images of a curve under its action (in blue), together with its fundamental domain, from which we obtain the
			surface $S'=S/\langle h \rangle$ by identifying the two dashed boundary components.}
		\label{fig:translatable}
	\end{figure}
	
	Let $X$ be a topological space, and let $G$ be a group. In this paper, we call
	an action $G \to \Homeo(X)$ \emph{properly discontinuous} if for every compact
	subset $K \subset X$, the set $\{g \in G : gK \cap K \neq \emptyset\}$ is finite.
	For a general discussion on various notions of proper discontinuity for group
	actions, we refer the reader to~\cite{Kap23}. The following lemma illustrates the
	structure of the end space of a translatable surface. We will use it to give
	examples of surfaces whose end space is doubly pointed but which are not translatable.

	\begin{lemma} \label{lem:translatable_Z-cover}
		Let $S$ be an infinite-type surface.
		\begin{enumerate}
			
			\item If $S$ is translatable, then for every translation $h$, the infinite cyclic group $\left\langle h\right\rangle$ acts properly discontinuously and freely on $S$. As a consequence, $S$ is a regular $\mathbb{Z}$--cover of the surface $S' \coloneqq S / \left\langle h \right\rangle$.
			\item If $S$ has a translation structure or hyperbolic structure and there exists an isometry $h$ of $S$ of infinite order, then the infinite cyclic group $\left\langle h\right\rangle$ acts properly discontinuously and freely on $S$. As a consequence, $S$ is a regular $\mathbb{Z}$--cover of the surface $S' \coloneqq S / \left\langle h \right\rangle$.
		\end{enumerate}
		
		\begin{proof}
			\begin{enumerate}
				\item Let $h$ be a translation in the sense of \cref{def:translatable}. Then
				$h$ has infinite order. We will show that the group $\left\langle h
				\right\rangle$ acts freely and properly discontinuously on $S$. Then $S'
				\coloneqq S / \left\langle h \right\rangle$ is a surface, see
				\cref{fig:translatable} for an example.
				
				Consider a compact subset $K$ of $S$. Assume there existed infinitely many
				elements $f \in \langle h \rangle$ with~$f (K) \cap K \neq \emptyset$. We
				may enlarge $K$ to a compact set that has negative Euler characteristic
				(because~$S$ is of infinite type), while maintaining this property.  We claim that $K$
				contains an essential closed curve~$\alpha$ such that, for infinitely many elements $f
				\in \langle h \rangle$, $f(\alpha)$ intersects $K$. Indeed, by hypothesis there are infinitely many integers  $n_i$, $i\in\N $, such that $ h^{n_i}(x_i)\in K$  for some~$x_i\in K $, depending on~$i$. Up to taking a subsequence (of the indexes $i\in\N$), we can suppose that the sequence~$(x_i)_{i\in\N}$ converges to a point $x$ in $K$. Let then $\alpha$ be an essential closed curve in~$K$ containing infinitely many elements of the sequence $(x_i)$.
				Then $f(\alpha)$ intersects $K$ for infinitely many $f \in \langle h \rangle$ and $\lim_{n
					\rightarrow \infty} h^n(\alpha)$ can neither be~$e_+$ nor $e_-$, which is a
				contradiction. Therefore the action of~$\langle h \rangle$ is properly
				discontinuous.
				
				Assume now that a nontrivial element $f \in \langle h \rangle$ fixes a point $x$. Then
				we can find a compact set~$K$ with negative Euler characteristic that contains
				$x$ and for all $n \in \mathbb{N}$, we have $f^n(K) \cap K \neq \emptyset$. By
				the argument in the preceding paragraph, this leads to a contradiction, and
				therefore the action is~free.
				
				\item We claim that, since $h$ is an isometry of infinite order, $\left\langle h \right\rangle$ acts properly discontinuously on~$S$. This follows from the fact that the isometry group of any infinite-type translation surface acts properly discontinuously on the surface, see
				\cref{cor:infinite_type_properly_discontinuously} in \cref{ssec:useful-results}, where we gather results about homeomorphism groups acting properly discontinuously. As in the previous case, we also obtain that $\left\langle h \right\rangle$ acts freely. Therefore, $S' \coloneqq S / \left\langle h \right\rangle$ is again a surface.
				\qedhere
			\end{enumerate}
		\end{proof}
	\end{lemma}
	
	\begin{remark}
		\label{rmk:translatable_Z-cover}
		As a consequence of \cref{lem:translatable_Z-cover}, for a translatable surface $S$ with a translation $h$, we have
		\begin{equation}
				\label{eq:decomposition-translatable-covering-0}
			S=\bigcup_{n\in\Z}h^n(F)
		\end{equation}
		for a fundamental domain $F$ of the action of $h$ on $S$. Since $S$ is connected, $F$ is a connected (possibly infinite-type) surface.  Indeed, the covering $S\to S'$ is defined by a surjective map $\varphi:\pi_1(S')\to\Z$. Using that $H^1(S';\Z)$ is homotopy equivalent to $[S', \mathbb{S}^1]$, the homotopy classes of maps $S' \to \mathbb{S}^1$, 
		we have that $\varphi$ is induced by a map $g \colon S' \to \mathbb{S}^1$. We can consider then a lift $\tilde{g}\colon S\to\R$ such that $\tilde{g}(h(x))=\tilde{g}(x)+1$ and take~$F=\tilde{g}^{-1}([0,1])$. Hence the end space of a
		translatable surface $S$ has a decomposition (which depends on the translation
		$h$) of the form
		\begin{equation}
			\label{eq:decomposition-translatable-covering}
			E(S) = \bigsqcup_{n \in \mathbb{Z}} E_n \sqcup \{e_+\} \sqcup \{e_-\}
		\end{equation}
		
		where:
		\begin{itemize}
			\item Each $E_n$ is homeomorphic to $E(S')$. Moreover, for every
			$n\neq m$, there exist open subsets $U_n^*,U_m^*$ of $E(S)$ such that
			$E_n\subset U_n^*$, $E_m\subset U_m^*$ and $U_m^*\cap U_n^*=\emptyset$.
			\item Both $e_+$ and $e_-$ are accumulation points of $\sqcup_{n \in
				\mathbb{Z}} E_n$. More precisely, we can assume that the
			enumeration $\sqcup_{n \in \mathbb{Z}} E_n$ is so that
			$h^k(E_n)=E_{n+k}$ for every $n,k\in\Z$, and that
			$\lim_{k\to\pm\infty}h^k(e)=e_\pm$ for every $e\in \sqcup_{n \in
				\mathbb{Z}} E_n$.
		\end{itemize}
		When $S$ has a translation structure or hyperbolic structure with an isometry of infinite
		order, such a decomposition of $E(S)$ is still valid but one could have that
		$e_+=e_-$. For example, consider the $\Z$--action on~$\C$ given by
		$h(z)=z+1$. Let $C\subset\C$ be a copy of the Cantor set contained in a
		fundamental domain of this~$\Z$--action and define $C_n\coloneqq h^n(C)$. Then
		$\C\setminus\sqcup_{n\in\Z}C_n$ has a natural translation structure for
		which $h$ is an isometry of infinite order. The end space of
		$\C\setminus\sqcup_{n\in\Z}C_n$ is a Cantor set. However, $h$ is \emph{not} a
		translation in the sense of \cref{def:translatable} because there, we
		require $e_+\neq e_-$.
	\end{remark}
	
	\begin{lemma} \label{lem:translatable_non_self-similar_implies_doubly_pointed}
		Let $S$ be a translatable surface. Then $E(S)$ fulfills exactly one of the following two possibilities:
		\begin{enumerate}
			\item The end space $E(S)$ is self-similar.
			\item The end space $E(S)$ is doubly pointed.
		\end{enumerate}
	\end{lemma}
		
	\begin{proof}
		By \cref{rmk:translatable_Z-cover}, if $S$ is a translatable surface
		with translation $h$ then $E(S)$ decomposes as
		in~\cref{eq:decomposition-translatable-covering} and
		$\lim_{k\to\pm\infty}h^k(e)=e_\pm$ for every $e\in \sqcup_{n\in\Z}
		E_n$.
		
		We can read from this description that the only points in $E(S)$ whose
		$\Homeo(S)$--orbit can potentially be finite are $e_+$ and $e_-$.
		
		If the $\Homeo(S)$--orbits of $e_+$ and $e_-$ are actually finite, the only possible points in their orbit(s) are $e_+$ and $e_-$ themselves. Hence for an open subset to be homeomorphic to $E(S)$, it has to contain two distinct points, locally homeomorphic to $e_+$ and $e_-$. However, there exists a finite clopen partition of $E(S)$ that separates $e_+$ and $e_-$, hence $E(S)$ cannot be self-similar. Furthermore, $E(S)$ is by definition
		doubly pointed and hence (2) but not (1) is fulfilled.
		
		If the $\Homeo(S)$--orbit of at least one of $e_+$ or $e_-$ is
		infinite (say of $e_+$), then $E(S)$ is not doubly pointed.
		In particular, the orbit of $e_+$ intersects one, and hence all of the sets $E_n$.
		In every finite clopen partition of $E(S)$, the clopen set which contains $e_-$
		has to contain infinitely many of the sets $E_n$, hence by \eqref{eq:decomposition-translatable-covering} a neighbourhood of $e_-$ of the form $\sqcup_{n \in \mathbb{N}} E_n \sqcup \{e_-\}$.
		Therefore, it contains an open set homeomorphic to $E(S)$,
		which implies that $E(S)$ is self-similar. Hence (1) but not (2) is fulfilled.
		\end{proof}
		
	There are uncountably many examples of surfaces that satisfy
	\cref{lem:translatable_non_self-similar_implies_doubly_pointed} (1). One example is
	$\mathbb{S}^2$ punctured at a Cantor set (a.k.a.\ the Cantor tree). Another one
	is the Cantor tree punctured at an infinite countable discrete set $U$ of points
	that accumulates to the Cantor set. More examples can be produced by replacing
	each point in $U$ in the preceding sentence by an ordinal of the form
	$\omega^n+1$ for some fixed $n\in\N$.
	
	In the case that the end space is countable, we can use the language of characteristic systems to characterize translatable surfaces.

	\begin{lemma} \label{lem:translatable_countable_case}
		Let $S$ be a surface whose end space has characteristic system $(\alpha, 2)$ and with $E(S) = E^g(S)$. Then $S$ is translatable if and only if $\alpha$ is a successor ordinal or zero.
		
		\begin{proof}
			Note first that from the Cantor--Bendixson degree being $2$, it follows that $S$ has doubly-pointed end~space.
			
			If $S$ is translatable, then the surface $S'$ from \cref{lem:translatable_Z-cover} has an end space with characteristic system~$(\alpha', d)$ for some $\alpha$ and $d$. The decomposition of $E(S)$ from \cref{rmk:translatable_Z-cover} implies then that $\alpha = \alpha'+1$ is a successor ordinal.
			
			Now let $\alpha$ be a successor ordinal or zero. If $\alpha = 0$, then $S$ is a cylinder and hence translatable. If $\alpha$ is a successor, then there exists $\alpha'$ with $\alpha = \alpha'+1$. Let $S'$ be a surface such that the characteristic system of $E(S')$ is $(\alpha', 1)$ and such that all ends are non-planar. Then the regular $\Z$--cover of $S'$ is translatable by definition and has characteristic system $(\alpha, 2)$, hence is homeomorphic to $S$. Therefore, $S$ is translatable.
		\end{proof}
	\end{lemma}
	
	We now present the topological trichotomy that is needed for the proof of our main result. For this, we need to recall the notion of non-displaceable subsurface.
	
	\begin{definition}
		\label{def:non-displaceable-subsurface}
		A subsurface $S'\subset S$ is called \emph{non-displaceable} if for every $f\in\Homeo(S)$, we have that~$f(S')\cap S'\neq\emptyset$.
	\end{definition}
	
	See \cref{fig:non-displaceable} for an example of a non-displaceable subsurface.

	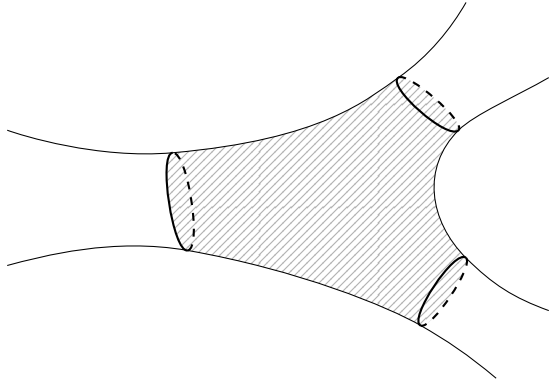
\begin{figure}[tbh]
		\centering
		\begin{tikzpicture}
			% \draw[color=gray!50] (-5,-2) grid (5.5,4.5); % help lines
			% non-displaceable subsurface
			\path[pattern color=white!70!black, pattern=north east lines] (-0.5,2) .. controls +(-175:0.2cm) and +(-170:0.2cm) .. (-0.3,0.7)
			.. controls +(-10:1.5cm) and +(150:0.5cm) .. (2.8,-0.3)
			.. controls +(-20:0.15cm) and +(-45:0.15cm) .. (3.4,0.6)
			.. controls +(135:0.7cm) and +(-135:0.7cm) .. (3.3,2.3)
			.. controls +(45:0.15cm) and +(35:0.15cm) .. (2.5,3)
			.. controls +(-145:0.5cm) and +(5:2cm) .. (-0.5,2);
			
			% surface
			\draw (-2.7,2.3) .. controls +(-20:0.2cm) and +(-175:1cm) .. (-0.5,2)
			.. controls +(5:2cm) and +(-145:0.5cm) .. (2.5,3)
			.. controls +(35:0.5cm) and +(-120:0.4cm) .. (3.4,4);
			\draw (-2.7,0.5) .. controls +(15:0.5cm) and +(170:1cm) .. (-0.3,0.7)
			.. controls +(-10:1.5cm) and +(150:0.5cm) .. (2.8,-0.3)
			.. controls +(-30:0.4cm) and +(140:0.5cm) .. (3.8,-1);
			\draw (4.5,3) .. controls +(-150:0.5cm) and +(45:0.3cm) .. (3.3,2.3)
			.. controls +(-135:0.7cm) and +(135:0.7cm) ..  (3.4,0.6)
			.. controls +(-45:0.5cm) and +(160:0.5cm) .. (4.5,-0.1);
			
			% boundary of non-displaceable subsurface
			% left
			\draw[thick, dashed] (-0.3,0.7) .. controls +(10:0.2cm) and +(5:0.2cm) .. (-0.5,2);
			\draw[thick] (-0.5,2) .. controls +(-175:0.2cm) and +(-170:0.2cm) .. (-0.3,0.7);
			% top
			\draw[thick] (2.5,3) .. controls +(-145:0.2cm) and +(-135:0.2cm) .. (3.3,2.3);
			\draw[thick, dashed] (2.5,3) .. controls +(35:0.15cm) and +(45:0.15cm) .. (3.3,2.3);
			% bottom
			\draw[thick] (3.4,0.6) .. controls +(135:0.2cm) and +(160:0.2cm) .. (2.8,-0.3);
			\draw[thick, dashed] (3.4,0.6) .. controls +(-45:0.15cm) and +(-20:0.15cm) .. (2.8,-0.3);
		\end{tikzpicture}
		\caption{A non-displaceable subsurface.}
		\label{fig:non-displaceable}
	\end{figure}
	
	In the proof, we use the Mann--Rafi order on the end space of a surface and facts about the maximal elements w.r.t.~this order.
	
	\emph{Mann-Rafi order on $E(S)$}.   For $x,y\in E(S)$, we say that
	$x\preccurlyeq y$ if for every neighborhood $U$ of~$y$, there exists a
	neighborhood $V$ of $x$ and $f\in\Homeo(S)$ such that $f(V)\subset U$.  If
	$x\preccurlyeq y$ and $y\preccurlyeq x$, we say that~$x\sim y$.
	The relation $\preccurlyeq$ defines a strict partial order $\prec$ on the set of equivalence classes under $\sim$. We denote the equivalence class of $x$ by $E(x)$.
	Mann and Rafi show that $\prec$ has maximal elements.
	The set of ends whose equivalence class is maximal with respect to  $\prec$ is denoted by $\mathcal{M}$. Note that $\mathcal{M}\subset E(S)$ is $\Homeo(S)$--invariant by definition.
	We refer the reader to~\cite{Mann-Rafi_23}*{Section~4} for more details.

	\begin{lemma}[Trichotomy of surfaces with all ends accumulated by genus]
		\label{lemma:trichotomy-of-surfaces}
		Let $S$ be an infinite-type surface such that $E(S) = E^g(S)$. Then $S$ falls into exactly one of the following categories:
		\begin{enumerate}
			\item $S$ has self-similar end space,
			\item $S$ is translatable and $E(S)$ is not self-similar,
			\item $S$ has a non-displaceable subsurface of finite type or has doubly-pointed end space but is not translatable.
		\end{enumerate}
		
		\begin{proof}
			If $E(S) = E^g(S)$ and this space is self-similar then every finite-type subsurface $S'\subset S$ is displaceable. The same is true if~$S$ is translatable (but not a cylinder). Furthermore, no doubly-pointed end space is self-similar. These three arguments are sufficient to see that the three categories above are disjoint. We now prove that every surface falls into one of them.
			
			For this, we show that every surface has self-similar end space or doubly-pointed end space or a non-displaceable subsurface of finite type by using two tools from \cite{Mann-Rafi_23}.
			The first is [\emph{Ibid.}, Proposition 4.8] which states that, under the hypothesis that every finite-type subsurface is displaceable, $E(S)$ is self-similar if and only if the set $\mathcal{M}$ of maximal ends is either a singleton or a Cantor set of equivalent points.
			The second are the following two criteria for the existence of a finite-type non-displaceable subsurface (see~[\emph{Ibid.}, Example~2.5]):
		
\begin{enumerate}
	\item $S$ has an infinite end space and there exists a $\Homeo(S)$-invariant, finite set $Z\subset E(S)$ of cardinality at least $3$.
	\item $E(S)$ contains two $\Homeo(S)$--invariant disjoint subsets $X$ and $Y$ with $X$ homeomorphic to the Cantor set. 
\end{enumerate}

By [\emph{Ibid.}, Proposition 4.7], $\mathcal{M}$ decomposes into equivalence classes that are either finite or Cantor sets. We go now through all the possibilities:

\begin{itemize}
	\item If $|\mathcal{M}| = 1$, then $E(S)$ is self-similar or $S$ contains a non-displaceable subsurface of finite type.
	\item If $|\mathcal{M}| = 2$, then $E(S)$ is doubly pointed since by definition, any end which is not maximal has an infinite $\Homeo(S)$--orbit. 
	\item If $\mathcal{M}$ contains more than two points which are contained in finite equivalence classes, then Criterion~(1) above implies that we obtain a non-displaceable subsurface.
	\item If $\mathcal{M}$ consists exactly of one infinite equivalence class, then this is a Cantor set. Hence $E(S)$ is self-similar or $S$ contains a non-displaceable subsurface of finite type.
	\item If $\mathcal{M}$ consists of an infinite equivalence class (which must be a Cantor set) and at least one other equivalence class, then Criterion (2) above implies that there exists a non-displaceable subsurface of finite type.
	\qedhere
\end{itemize}
\end{proof}
\end{lemma}

\subsection{Two results on homeomorphism groups acting properly discontinuously}
\label{ssec:useful-results}

In the following, we present two technical results that are crucial for the proof of our main results.

\begin{lemma}
	\label{thm:topological-restrictions-homeo}
	Let $G<\Homeo(S)$ be a group acting properly discontinuously on a surface $S$. Then:
	\begin{enumerate}
		\item $G$ is countable.
		\item If $S$ has a non-displaceable subsurface of finite type, then $G$ is
		finite.
		\item If $S$ has doubly pointed end space but is not a cylinder, then $G$ is virtually cyclic.
	\end{enumerate}
\end{lemma}

\begin{proof}
	We first show that $G$ is countable. If $S$ is compact, then $G$ is finite, so
	we may assume that $S$ is non-compact. Let $\cup_{i\in\N} K_i=S$ be an
	exhaustion of $S$ by compact subsets. For every $K_i$, let $G_i = \{g : gK_i\cap
	K_i\neq\emptyset\}$. Then $\cup_{i\in\N}G_i$ is countable. If $G$ is uncountable
	there exists $g\in G\setminus \cup_{i\in\N}G_i$. Then, for every~$x\in K_1$, we
	have that $gx\notin K_i$ for all $i\in\N$, but this is a contradiction because
	$\cup_{i\in\N} K_i=S$. We conclude that $G$ is countable.
	
	Next, suppose that $S$ has a non-displaceable subsurface of finite type. Then
	we can choose a compact set~$K$ that includes this non-displaceable subsurface. For every element $g\in G$, we have $g(K)\cap K\neq\emptyset$, hence~$G$ cannot be infinite.
	
	Finally, we assume that $S$ is non-compact and has doubly pointed end space.
	By~\cite{APV}*{Theorem~4.13}, if $S$ is a complete hyperbolic surface and $G =
	\Isom(S)$, then $G$ is virtually cyclic. Their proof only uses topological
	arguments and the fact that $G$ acts properly discontinuously on $S$, so the
	proof of (3) is essentially the same. We include a sketch of the argument here
	for completeness.
	
	Since every finite group is virtually cyclic, we assume henceforth that $G$ is
	infinite. Let $\{e_1,e_2\}$ be the two distinguished elements of $E(S)$ and let
	$\Lambda$ be the stabilizer of $e_1$ in $G$. This is a subgroup of $G$ of index
	at most $2$ and hence itself infinite. The closure of any $\Lambda$--orbit in
	$E(S)$ contains $\{e_1,e_2\}$. Indeed, this is~\cite{APV}*{Lemma~4.14}. The
	proof of this lemma is again purely topological,\footnote{As a matter of fact,
		this proof uses at some point the existence of a separating simple closed
		geodesic $\gamma$, but it remains valid if we change $\gamma$ for a separating simple
		closed \emph{curve}.} so we refer to [\emph{Ibid}.] for details.
	There are two cases to consider:
	\begin{itemize}
		\item First suppose that $S$ has infinitely many ends. Let $e$ be an end
		with an infinite $G$--orbit. Such an end exists because if not there would
		be a finite-type $G$--non-displaceable subsurface and this would contradict $G$
		being infinite. Then the closure of the orbit $\Lambda e$ contains both
		$e_1$ and $e_2$. Let $\gamma$ be a simple closed curve separating $e_1$ from
		$e_2$, and denote by $L_\gamma\subset\Lambda e$ the set of ends to the left
		of $\gamma$. Define $\phi\colon \Lambda\to\Z$ by:
		\[
		\phi(g)=|g^*(L_\gamma)\setminus L_\gamma|-|L_\gamma\setminus g^*(L_\gamma)|.
		\]
		The map $\phi$ is a group homomorphism that measures the net change of how
		many ends in $\Lambda e$ move from the left of $\gamma$ to the right of
		$\gamma$. It has finite kernel, and since $\Lambda$ has index at most $2$ in
		$G$, we conclude that $G$ is virtually cyclic.
		
		\item If $S$ has finitely many ends, then the hypothesis on all finite-type
		subsurfaces to be displaceable, implies that $S$ has at most $2$ ends.
		By hypothesis $S$ is not a cylinder, hence the only possibilities are that $S$ is homeomorphic to either the Loch Ness monster punctured at
		one point or to Jacob's Ladder\,\footnote{The Jacob's Ladder is the orientable
			topological surface of infinite genus and whose end space consists exactly
			of two non-planar ends.}. In the former case, any finite-type subsurface containing the
		puncture would be~$G$--non-displaceable, leading to a contradiction. In the
		latter case, we can again construct a group homomorphism $\phi\colon \Lambda\to\Z$
		with finite kernel. For details, see~\cite{APV}*{Proposition~3.3}. \qedhere
	\end{itemize}
\end{proof}

We finish this section with an important result on Riemann surfaces that has
implications for translation surfaces with non-finitely generated fundamental
groups. Specifically, we obtain that the isometry group of any such translation
surface $M$ acts properly discontinuously on $M$. This result is a consequence
of the following theorem, which characterizes Riemann surfaces whose group of
analytic automorphisms does not act properly discontinuously.

\begin{theorem}
	\label{thm:BigAnalyticAut}
	Let $X$ be a Riemann surface such that the group $\Aut(X)$ of its analytic
	automorphisms does not act properly discontinuously on $X$. Then $X$ is
	conformally equivalent to one of the following surfaces:
	\begin{enumerate}
		\item the Riemann sphere $\widehat{\C}$,
		\item the plane $\C$,
		\item the half-plane $\mathbb{H}^2$,
		\item an annulus $\{z \in \C : r < |z| < R\}$, with $0 < r < R \leq \infty$,
		\item the punctured plane $\C^*$, or
		\item a torus $\R^2 / \Lambda$.
	\end{enumerate}
\end{theorem}

The proof follows from the Poincar\'e--Koebe uniformization theorem. A simple
proof can be consulted in~\cite{DHV1}*{Chapter~3}. As a consequence of this
theorem and~\cite{APV}*{Lemma~2.6} we deduce the following:

\begin{corollary} \label{cor:infinite_type_properly_discontinuously}
	The isometry group of any infinite-type translation or hyperbolic surface $M$ acts properly discontinuously on $M$.
\end{corollary}

\section{Geometrical preliminaries: translation surfaces}

\label{sec:transsurf-preliminaries}
This section
provides basic definitions, examples, and geometric invariants of translation
surfaces, aimed at readers unfamiliar with the topic. We do not assume that $S$
is a finite-type topological surface.

Translation surfaces with non-finitely generated fundamental groups are natural
in the study of classical dynamical systems such as polygonal billiards,
periodic wind-tree models, or baker's maps. For a more detailed discussion of
these examples and the definitions presented here, readers are referred to
Chapter 1 of~\cite{DHV1}.

We begin with a geometrical definition in the spirit of $(G,X)$--structures
\emph{à la Thurston}. A \emph{translation atlas} is an atlas on $S$ where all
transition functions are translations of the plane. Every surface with a
translation atlas can be endowed with a flat metric by pulling back the
(translation invariant) Euclidean metric $dz$ in~$\C$. Consider the Riemannian
metric $(\C^*,g_\alpha)$ given in polar coordinates by
$g_\alpha \coloneqq (dr)^2+(\alpha rd\theta)^2$, $\alpha>0$. A
point $z$ in a surface $S$ endowed with a Riemannian metric $\mu$ is called a
\emph{conical singularity} of angle $2\pi\alpha$ if there exists a neighborhood
of $z$ isometric to a neighborhood of the origin in $(\C^*,g_\alpha)$. Conical
singularities are unavoidable when $S$ is compact of genus greater than 1.

\begin{definition}[Geometric definition of a translation surface]
	\label{def:geometrical-TS}
	A translation surface $M$ is a pair $(S,\mathcal{T})$ consisting of a connected
	topological surface $S$ and a maximal translation atlas on $S\setminus\Sigma$,
	where $\Sigma$ is a discrete subset of $S$ such that every $z\in\Sigma$ is a
	\emph{conical singularity}. The maximal translation atlas $\mathcal{T}$ is
	called a \emph{translation structure} on $S$, and its charts are referred to as
	\emph{translation charts}.
\end{definition}

Since the natural flat metric in a translation surface has constant curvature 0
in the complement of $\Sigma$, all conical singularities have an angle of $2\pi
k$ for some positive integer $k > 1$.

\begin{remark}
	\label{rmk:defs-TS}
	There are two other equivalent ways to define what a translation surface is: a
	constructive definition and an analytical definition. In the constructive
	definition, a translation surface is obtained by gluing countably many
	Euclidean polygons along parallel sides of the same length using translations.
	In the analytical definition, a translation surface is a pair $(X,\omega)$ where
	$X$ is a Riemann surface and~$\omega$ is a non-identically zero holomorphic
	$1$--form. A proof of the equivalence of the three definition can be found
	in~\cite{Wright15}. In this text, we adopt the definition that is more
	convenient in each context, which in almost all cases will be
	\cref{def:geometrical-TS} or the constructive one.
\end{remark}

\subsection{Affine maps and Veech groups}

Let $M_1$ and $M_2$ be translation surfaces with conical singularities $\Sigma_i
\subset M_i$, for $i=1,2$. We call a map $f\colon  M_1 \to M_2$ \emph{affine} if, in
translation charts, the restriction $f|\colon  M_1 \setminus \Sigma_1 \to M_2
\setminus \Sigma_2$ is an $\R$--affine map and $f(\Sigma_1) \subset \Sigma_2$.
We denote the group of affine automorphisms of $M$ \emph{that preserve
	orientation} by $\Aff(M)$.

Since the transition functions in $M \setminus \Sigma$ are translations, the
derivative $Df$ of an affine map $f$ is constant. This implies the existence of a well-defined
derivative group homomorphism
\[
Df\colon \Aff(M)\to\GL^+(2,\R).
\]
We define the subgroups $\Trans(M) \coloneqq Df^{-1}(Id)$ and $\Isom(M) \coloneqq
Df^{-1}(\SO(2,\R))$. The affine maps in $\Trans(M)$ and $\Isom(M)$ are called
\emph{translations} and \emph{isometries}, respectively. The image of $Df$ in
$\GL^+(2,\R)$ is called the \emph{Veech group} $\Gamma(M)$ of $M$. We have an exact
sequence
\[
1\to\Trans(M)\to\Aff(M)\to \Gamma(M)\to 1.
\]

\begin{lemma}
	\label{lemma:isom-into-map}
	Let $S$ be a surface which has either negative Euler characteristic or is of infinite type. Then for any translation structure on $S$, the natural map $\Isom(M)\to\Map(S)$, that assigns to each isometry of~$M$ its
	mapping class, is an injective homomorphism.
\end{lemma}

\begin{proof}
	The translation surface $M$ has a natural Riemann surface structure. Indeed,
	there is a unique analytic extension of the translation structure of
	$M\setminus\Sigma$ to a Riemann surface structure $R_M$ on $M$. Moreover,
	$\Isom(M)< \Aut(R_M)$, where the latter denotes the group of conformal
	automorphisms of $R_M$. To conclude, note that every conformal automorphism
	of $f\in \Aut(R_M)$ which is isotopic to the identity is actually
	$\Id_{R_M}$. This is because a (properly normalized) lift of $f$ to the
	universal cover $\widetilde{R_M}$ extends to the boundary by the identity on
	the limit set. Therefore, the lift of $f$ is a Möbius transformation that
	fixes three points on the boundary.
\end{proof}

As a corollary, we get that the natural map $\Trans(M)\to\Map(S)$ that assigns
to each translation of $M$ its mapping class, is an injective homomorphism.

\subsection{Translation flow, surgeries, and linear actions}
\label{sssec:flow-surgeries-action}
For any translation surface $M$ and $\theta \in \R/2\pi\Z$, the constant vector
field $e^{i\theta}$ in $\C$ can be pulled back to a constant vector field
$X_\theta$ on $M \setminus \Sigma$. For any $z \in M \setminus \Sigma$, let
$\gamma_z\colon  I \to M$ be the maximal integral curve of $X_\theta$ with initial
condition $\gamma_z(0) = z$. The map~$t \to \gamma_z(t)$ defines a local action
of $\R$ on $M$ called the \emph{translation flow} on $M$ in direction $\theta$.
A partial or total orbit of the translation flow is called a \emph{geodesic} of
$M$. If $I$ has finite length, then we call $\gamma_z$ or its image in~$M$ a
\emph{saddle connection} in direction $\theta$. Typically, a saddle connection
is a trajectory of the translation flow joining (not necessarily distinct)
points in $\Sigma$. If $\gamma$ is a saddle connection in direction $\theta$ of
length $|\gamma|$, then the vectors~$\pm|\gamma|e^{i\theta}\in\C$ are called
the \emph{holonomy vectors associated with $\gamma$}.

To prove our main results, we use a particular surgery on translation surfaces
called \emph{gluing along slits}. We describe this procedure in what follows.
Let $\gamma_i \subset M_i$, $i = 1,2$, be two parallel geodesic segments of the
same length. The gluing of $M_1$ to $M_2$ along the pair ${\gamma_1, \gamma_2}$
is the translation surface $M'$ obtained by cutting~$M_i$ along $\gamma_i$ and
then gluing them back by identifying, using a translation, the left side of
$\gamma_1$ to the right side of~$\gamma_2$ and vice versa. This procedure is
illustrated in \cref{fig:slits}.

\begin{figure}[tb]
	\includegraphics[scale=.8]{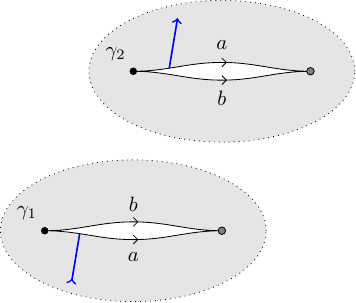}
	\caption{Gluing along two slits with a geodesic depicted in blue.}
	\label{fig:slits}
\end{figure}
The translation surface $M'$ inherits from the gluing two conical singularities
of angle $4\pi$ which appear as extremities of a saddle connection parallel to
$\gamma_i$ of length $|\gamma_i|$.

\begin{remark}
	\label{rmk:gluing-infinite-slits}
	In the proof of our main results, we consider the following cases or variations
	of the aforementioned surgery:
	\begin{enumerate}
		\item  Let $\gamma_1$ and $\gamma_2$ be two infinite parallel rays in two different copies $P_1$ and $P_2$ of the Euclidean plane, both departing in the respective origins. The gluing of $P_1$ and $P_2$ along $\gamma_1$ and
		$\gamma_2$ is defined analogously as when~$\gamma_1$ and $\gamma_2$ are finite-length geodesic segments.
		The resulting surface~$M$ is by construction a $2:1$ covering of the Euclidean plane,
		ramified over the origin. In particular, $M$ is homeomorphic to the
		Euclidean plane.	
		\item The gluing of two copies $P_1$ and $P_2$ of the Euclidean plane along
		two finite-length slits results in a translation surface $M$ that is
		homeomorphic to an annulus.
		\item The gluing of two copies $P_1$ and $P_2$ of the Euclidean plane along
		two infinite families $\{\gamma_{i,j}\}_{j\in\N}$ of parallel slits of the
		same length, with the property that for each $i=1,2$ the set
		$\{\gamma_{i,j}\}_{j\in\N}$ is a closed subset of the plane, is homeomorphic
		to the Loch Ness monster.
	\end{enumerate}
\end{remark}

The group $\GL^+(2,\R)$ acts
naturally on the set of all translation surfaces: If $(S,\mathcal{T})$ is a
translation surface as in the geometric definition,
$\mathcal{T}=\{\phi_i\colon U_i\to\C\}_{i\in I}$ and $A\in\GL^+(2,\R)$, then
$A(S,\mathcal{T})=(S,A\mathcal{T})$, where $A\mathcal{T}=\{A\circ\phi_i\colon U_i\to\C\}_{i\in I}$. Here $A\circ\phi_i$ denotes
the \emph{post}-composition of $\phi_i$ with the linear transformation of the
plane defined by the matrix $A$. Note that $A\mathcal{T}$ is indeed a
translation atlas because:
\[
((A\circ\phi_j)\circ(A\circ\phi_i)^{-1})(z) = A\circ(\phi_j\circ\phi_i^{-1})(A^{-1}(z)) = A(A^{-1}(z)+c) = z + A(c)
\]
In the context of finite-type translation surfaces, this action is a fundamental
ingredient for the study of the dynamics of the translation flow. For a detailed
discussion of this action in the context of finite-type surfaces, see Section 3
in \cite{Wright15}.

\subsection{The end-grafting construction}
\label{ssec:Maluendas-Randecker}

In this section, we present the end-grafting construction, which is used to
prove several results in \cref{sec:proofs-isometries,sec:proofs-Veech}.
We begin with a more abstract statement on the existence of translation structures on most Riemann surfaces.

\begin{proposition}
	\label{prop:translation_structure_Riemann_surface}
	Every orientable surface $S$ which is not homeomorphic to the sphere $\mathbb{S}^2$ admits a translation structure.
	\begin{proof}

		Any surface $S$ can be triangulated and hence admits a Riemann surface structure
		$X=X(S)$, see Section 46A in~\cite{AhlforsSario60}. It is then sufficient to
		show that if $X$ is not homeomorphic to the sphere $\mathbb{S}^2$, then~$X$
		admits a non-identically zero holomorphic $1$--form $\omega$. If $X$ is compact
		of genus $0<g<\infty$ then the space of holomorphic $1$--forms is a complex
		vector space of dimension $g$. If $X$ is non-compact then, as known
		from complex geometry, any holomorphic vector bundle on $X$ is holomorphically
		trivial. In particular, it is possible to extract a non-vanishing section
		$\omega$ of the canonical bundle on $X$ (that is, a never vanishing
		holomorphic~$1$--form). The desired translation structure is $(X,\omega)$ and the
		charts are given by local integration of the~$1$--form~$\omega$.
	\end{proof}
\end{proposition}

The end-grafting construction described by Maluendas and Randecker
guarantees certain geometric constraints that
\cref{prop:translation_structure_Riemann_surface} can not certify. We now
provide a brief overview of this construction; see Theorem 2
in~\cite{Randecker2016} for a detailed description.

The idea is the following: Given two closed, non-empty subsets $E'\subset E$ of the Cantor
set, we construct a translation surface $M$ such that $E^g(M)\subset E(M)$ is homeomorphic to
$E'\subset E$.
To achieve this, we start with a rooted infinite tree $T_E$ with end space
homeomorphic to $E$, where each vertex has either degree $2$ or~$3$. For
illustration purposes, we assume that each edge has length 1. We then define a
countable family of rays in $T_E$ which will be used to construct $M$.  These rays cover the whole tree $T_E$ and are disjoint except in the vertices, see \cref{fig:MR-construction} for illustration.

\begin{figure}
	\centering
	\def\svgwidth{.9\textwidth}
	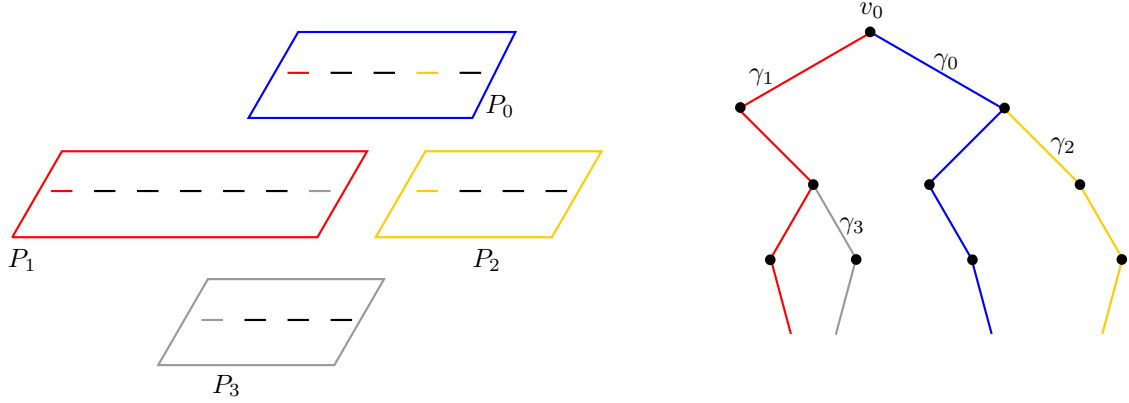
	\caption{Some of the gluings in the end-grafting construction.}
	\label{fig:MR-construction}
\end{figure}

Let $v_0$ be the root of $T_E$. We begin with an infinite injective ray $\gamma_0\colon [0,\infty)\to T_E$ such that
$\gamma_0(0)=v_0$ and $\{\gamma_0(0), \gamma_0(1), \gamma_0(2), \dots\}$ is the set of vertices on~$\gamma_0$.
For every $n\geq 1$,
we choose an edge $e_n$ in the complement of $\cup_{i=0}^{n-1}\operatorname{im}(\gamma_i)$ that is closest to $v_0$.
Let $v_n$ be the vertex of $e_n$ which is closer to $v_0$ and choose an infinite ray
$\gamma_n\colon [0,\infty)\to T_E\setminus\cup_{i=0}^{n-1}\operatorname{im}(\gamma_i)$ starting at
$v_n$ such that $\{v_n=\gamma_n(0), \gamma_n(1), \gamma_n(2), \dots\}$ is the set of vertices on~$\gamma_n$.
Proceeding inductively, we
obtain a countable family $\Gamma_E \coloneqq \{\gamma_n\}_{n\geq 0}$ whose union is~$T_E$.

To each $\gamma_n\in\Gamma_E$, we associate a Euclidean plane $P_n$ with
coordinates $(x,y)$. We define the following families of disjoint horizontal
slits of length 1 in $P_n$:
\begin{itemize}
	\item $s(n,k)$ is the slit with extremities $(6k,0)$ and $(6k+1,0)$,
	\item $t(n,k)$ is the slit with extremities $(6k+2,0)$ and $(6k+3,0)$, and
	\item $t'(n,k)$ is the slit with extremities $(6k+4,0)$ and $(6k+5,0)$.
\end{itemize}
where $k$ ranges over all positive integers. Now perform the following gluings
on $\sqcup_{n\geq 0} P_n$:
\begin{itemize}
	\item By construction, for every $n\geq 1$, there exists a unique $n'\geq 0$
	such that $\gamma_n(0)=\gamma_{n'}(l)$ for some~$l\in\N$. Then glue the slit
	$s(n,0)$ to $s(n',l)$. Forget all slits $s(n,k)$ which are not glued. The result
	is a genus zero translation surface $M'$. If $E^g=\emptyset$, the construction
	stops here.
	
	There is an embedding $i\colon T_E\hookrightarrow M'$. Indeed, let $T_n \colon  \gamma_n \to
	[0,\infty)$ be the linear map whose image is the $x$--axis of $P_n$ and that
	sends for each $n\geq 0$ the edge $\gamma_n(0)\gamma_n(1)$ to the segment
	$[0,6]\subset P_n$. Then the family $\{T_n\}_{n\geq 0}$ can be glued in a
	natural way together to define $i$. Note that $i$ covers all glued slits and is
	a homeomorphism at the level of end spaces.
	\item For each $e\in E'$, let $\gamma_e\colon [0,\infty)\to T_E$ be an infinite
	injective ray starting at $v_0$, such that $\gamma_e^{-1}(V(T_E))=\Z_{\geq 0}$,
	and whose end is precisely $e$. We glue $t(n,k)$ to $t'(n,k)$ whenever these
	slits are contained in the image of $i(\gamma_e)$. We denote by $M$ the
	resulting translation surface obtained by these gluings. We forget all slits
	$t(n,k)$ and $t'(n,k)$ which are not glued.
\end{itemize}

As shown in~\cite{Randecker2016}, the resulting translation surface $M$ has the
property that $E^g(M)\subset E(M)$ is homeomorphic to $E'\subset E$. Moreover, the end-grafting construction proves the following result:

\begin{proposition}
	\label{prop:MR-construction}
	Let $E' \subset E$ be two closed, non-empty subsets of the Cantor set. Then there exists a translation surface $M$ such that $E^g(M) \subset E(M)$ is homeomorphic to $E' \subset E$.
	Moreover, we can choose $M$ such that it has the following properties:
	\begin{enumerate}
		\item All singularities have cone angle equal to $4\pi$.
		\item All saddle connections in $M$ are horizontal and have length $1$.
		\item  It is always possible to embed a closed half-plane
		$H_0\hookrightarrow P_0$ such that $\partial H_0$ does not intersect any of
		the slits used in the construction.
	\end{enumerate}
\end{proposition}

\section{Proof of main results regarding isometry groups}
\label{sec:proofs-isometries}
In this section, we present the proofs for our results concerning the
isometry groups of translation structures on infinite-type surfaces. We first
provide the proofs of the results with shorter arguments.

\begin{proof}[Proof of  \cref{thm:realization-isometries-doubly-pointed}]
	Recall that we want to show that, for any infinite-type surface $S$ whose
	end space is doubly pointed and any translation structure $M$ on
	$S$, the group $\Isom(M)$ is virtually cyclic. From
	\cref{cor:infinite_type_properly_discontinuously}, we know that the isometry
	group $\Isom(M)$ acts properly discontinuously on $M$.  Since the end space
	of $M$ is doubly pointed, by \cref{thm:topological-restrictions-homeo} (3),
	we have that $\Isom(M)$ is virtually cyclic.
\end{proof}

\begin{proof}[Proof of \cref{thm:non-displaceable-implies-finite-isometries}]
	Recall that we want to show that any infinite-type translation surface $M$
	having a non-displaceable subsurface of finite-type has to have a finite group
	of isometries. As in the preceding proof, from
	\cref{cor:infinite_type_properly_discontinuously}, we know that the isometry
	group $\Isom(M)$ acts properly discontinuously on $M$. Hence, by
	\cref{thm:topological-restrictions-homeo}\,(2), $\Isom(M)$ must be finite.
\end{proof}

\subsection{Overview of the constructions}
\label{ssec:overview-constructions}
Before diving into the details of the proofs of the remaining results for
isometry groups, let us give an overview of the constructions we will use.

We begin by using \cref{thm:topological-restrictions-homeo}, which
establishes that $\Isom(M)$ is always countable. Next, we demonstrate that for
any infinite countable group $G$, there exists a translation structure~$M$ on
$S$ such that $\Isom(M)$ is isomorphic to $G$. Our approach is to modify a
variant of the construction outlined in~\cite{APV}*{Section~4.2} to the context
of translation surfaces, which we will now detail.

We start with a Cayley graph $C_G$ of $G$ and a translation surface~$M_{\Id}$
that corresponds to the identity $\Id \in G$. This graph serves as a
\emph{blueprint} for the subsequent steps. For every vertex $g\in G$ of $C_G$,
we consider a (potentially altered) copy $M_g$ of $M_{\Id}$. Then, using
translations and following specific rules depending on the edges of $C_G$, we
glue the elements in $\sqcup_{g\in G} M_g$ along slits. The resulting
translation surface $M$ is without boundary. The end space of $M$ and $\Isom(M)$
are determined by~$M_{\Id}$ and the gluing process.

Note that the specific details of the construction are contingent on the desired
outcome.

\subsection{Proof of \cref{thm:realization-isometries-selfsimilar}}
\label{ssec:proof-realization-isometries-selfsimilar}

We present first a detailed constructive proof which will be later
extended to obtain our main results about Veech groups. At the end of this
section, we present an argument to deduce
\cref{thm:realization-isometries-selfsimilar} from~\cite{APV}*{Theorem~6.2}. We thank the anonymous referee who provided the generalities of this argument.

In the context of \cref{thm:realization-isometries-selfsimilar}, the
general construction outlined in \cref{ssec:overview-constructions} takes the
following form.

\subsubsection*{Blueprint} Let $G$ be a countable group and $C_G$ be the \emph{complete Cayley graph} of $G$. That is, the vertex set~$V(C_G)$ of $C_G$ consists of all elements of $G$, and
for each ordered pair $(g,h)\in G^2$ with $h\neq \Id$, there is a directed edge
from~$g$ to $gh$ labeled by $h$. Since $C_G$ is a complete, labeled Cayley graph, the
action by left multiplication of $G$ on itself defines a representation $G\to\Aut(C_G)$ which is an isomorphism.

\subsubsection*{Vertex surface $M_{\Id}$} Let $x_\infty\in E(S)$ be the star
point of $E(S)$ (see \cref{rmk:star-point} and
\cref{thm:self-similar-radially-symmetric}). In particular, the set
$E(S)\setminus \{x_\infty\}$ can be decomposed as a disjoint union
$\sqcup_{n\in\mathbb{N}}E_n$, where $\{E_n\}_{n\in\N}$ is a collection of
pairwise homeomorphic subspaces of $E(S)$, and $x_\infty$ is in the closure of
$E_n$ for every $n$. Furthermore, since~$g(S)>0$ and every finite-type
subsurface in $S$ is displaceable, we have $g(S)=\infty$. Finally, as $E(S)$ is
self-similar, the star point $x_\infty$ belongs to $E^g(S)$.

Let $\overline{E_1} = E_1\sqcup\{x_\infty\}$ and $M_{\Id}'$ be the translation
surface obtained from the end-grafting construction applied to $ \overline{E_1} \cap
E^g(S)\subset\overline{E_1}$. We make the following convention: the ray
$\gamma_0$ used in the end-grafting construction corresponds to
$x_\infty\in\overline{E_1}$. We denote by $x_{\infty,\Id}$ the element of
$E(M_{\Id}')$ that corresponds to $x_\infty$.

By \cref{prop:MR-construction}, we can choose an upper half-plane $H_0$ in
$M_{\Id}'$ with a coordinate system $\mathbb{R} \times [0,\infty)$ which is
completely contained in $P_0$ and does not intersect any of the slits used in
the end-grafting construction.
We enumerate the elements of $G$ as $\{g_0, g_1, \dots\}$. For each element
$g_n$, we construct two vertical infinite rays in $H_0$ from $(4n,0)$ to
$x_{\infty,\Id}$ and from $(4n+1,0)$ to $x_{\infty,\Id}$, called
$s_{g_n,\text{in}}$ and $s_{g_n,\text{out}}$.

We continue by performing surgery on $M_{\Id}'$, in order to destroy all
possible non-trivial isometries by creating a distinguished conical
singularity.  Let $P$ be a convex $20$--gon representing a translation
surface~$M'$ with only one conical singularity $\sigma_{\Id}$. This means that
$M'$ is obtained from $P$ by identifying sides in $\partial P$ using
translations. Since a generic translation surface of genus 5 has trivial
isometry group, we can ensure that $\Isom(M')$ is trivial. Moreover, without
loss of generality, we can rescale~$P$ using a homothety so that the diameter of
$P$ is less than $1$.

Next, we consider an isometric embedding $i\colon P\hookrightarrow H_0$ such that the
image is at a large distance (for instance, a googol, which is a 1 followed by 100 zeroes) from $\sqcup_{n\geq
	0}\{{s_{g_n,\text{in}},s_{g_n,\text{out}}}\}$ and all $4\pi$--singularities
arising from the end-grafting construction. Removing the interior of~$i(P)$ from $H_0$ and
performing the edge identifications that define $M'$ on $M_{\Id}'$, we obtain a
new translation surface~$M_{\Id}$ with $\sigma_{\Id}$ as its only conical
singularity with total angle different from $4\pi$.

We claim $\Isom(M_{\Id})$ is trivial. Indeed, since $i(P)$ is a googol away from
all slits and singularities and $P$ is convex, the saddle connections in
$M_{\Id}$ defined by the sides of $P$ are the only ones of length less than~$1$
and hence form an $\Isom(M_{\Id})$--invariant set. Furthermore, the polygon $P$
can be chosen such that no two non-parallel sides have the same length. But then
any isometry would have to fix two non-parallel saddle connections of different
lengths. This implies that $\Isom(M_{\Id})$ is trivial.

\subsubsection*{Construction of the surface $M$ by gluings} For every $g\in G$,
let $M_g$ denote a copy of $M_{\Id}$. For every directed edge with label $h$
between the vertices $g$ and $gh$ in $C_G$, we glue  $s_{h,\text{out}}$ in $M_g$
to $s_{h,\text{in}}$ on~$M_{gh}$. We denote the resulting surface $M$.

We recall from \cref{rmk:gluing-infinite-slits} that the gluing of two copies of $\mathbb{R}^2$ along two infinite slits
is homeomorphic to another copy of~$\mathbb{R}^2$. By applying this fact each time we
perform one of the gluings defined above, we deduce that all the points
$\{x_{\infty,g}\}_{g\in G}$ merge into a single point $y_\infty \in E(M)$.

\begin{lemma}\label{lemma:isometries-are-correct}
	$\Isom(M)$ is isomorphic to $G$.
\end{lemma}

\begin{proof}
	By construction, there is a natural embedding $G\hookrightarrow \Isom(M)$
	which sends $g \mapsto T_g$ where $T_g | \colon M_h \to M_{gh}$
	is the identity in
	any local coordinate. Now, let $T\in\Isom(M)$. Then, there exists $g\in G$
	such that~$T_g^{-1}\circ T(\sigma_{\Id})=\sigma_{\Id}$. The map $T_g^{-1}\circ
	T$ fixes any saddle connection joining $\sigma_{\Id}$ to itself,
	which are precisely those saddle connections that correspond to the sides of
	$P$.
	
	Let $(U,\varphi)$ be a sufficiently small chart of $M_{\Id}$ around a point
	on one of the aforementioned fixed saddle connections. Then, in the
	$(U,\varphi)$ coordinate, we have that $T_g^{-1}\circ T$ is the identity.
	Since isometries are conformal maps, they are completely determined by the
	image of an open set. Therefore, $T=T_g$, and we have shown that $\Isom(M)$
	is isomorphic to $G$.
\end{proof}

\begin{lemma}\label{lemma:topology-is-correct}
	$M$ is homeomorphic to $S$.
\end{lemma}

\begin{proof}
	For every $g\in G$, we have an embedding $\psi_g\colon E(M_g)\hookrightarrow
	E(M)$. Since $E(M_g)$ is a closed subset of the Cantor set, the image of
	each $\psi_g$ is a closed subset of $E(M)$. As mentioned before,
	$\psi_g(E(M_g))\cap\psi_h(E(M_h))=y_\infty$ for every $g\neq h$. Hence all
	the $\psi_g$'s can be glued together to define an embedding
	$\psi\colon E=\sqcup_{n\in\mathbb{N}}E_n\sqcup\{x_\infty\}\hookrightarrow E(M)$
	sending $x_\infty$ to $y_\infty$. We claim that the image of $\psi$ is
	$E(M)$. This follows from the fact that the gluing process used to construct
	$M$ does not produce any new end. We can think of each $M_g$ as a copy of
	$\mathbb{R}^2$ with some extra topology localized inside some unbounded subset $C$
	. The gluing of $M_g$ to $M_{gh}$ occurs along slits
	disjoint from $C$. We create no new ends when gluing $M_g$ to $M_{gh}$ because
	the gluing of two copies of $\mathbb{R}^2$ along two infinite slits produces
	a surface that is homeomorphic to another copy of $\mathbb{R}^2$, see
	\cref{rmk:gluing-infinite-slits}. When taking the limit of the gluing
	process, the number of ends also does not increase, so the image of $\psi$ is
	$E(M)$.
\end{proof}

The two previous lemmas finish the proof of
\cref{thm:realization-isometries-selfsimilar}.
\qed
\medskip

\begin{remark}\label{rem:deduction-proof}
	
	As mentioned in the beginning of
	\cref{ssec:proof-realization-isometries-selfsimilar}, for surfaces $S$ with
	self-similar end spaces and no planar ends, one can also directly deduce
	\cref{thm:realization-isometries-selfsimilar} from~\cite{APV}*{Theorem~6.2}
	as follows. Let $S$ be such a surface, $G$ be a countable group, and $Y$ be
	the hyperbolic structure on~$S$ given by~\cite{APV}*{Theorem~6.2} whose
	isometry group is isomorphic to $G$. Abusing the notation, let us write
	$G=\Isom(Y)$. If $S$ has one end, the surface $Y$ is the surface $X^G_S$
	from \cite[Section 3.1]{APV} and the quotient manifold $Y/G$ is a Loch Ness
	monster by \cite[Lemma 3.4, Lemma 3.9]{APV} and thus in particular of
	infinite genus. Otherwise $Y$ is the surface $Y^G_S$ from  \cite[Section
	4.2]{APV}. In this case, $Y$ has infinitely many ends since the space of
	ends is self-similar and the quotient $Y/G$ is a non compact infinite genus
	surface.
	
	By a detailed inspection of the constructions carried out
	in~\cite{APV}*{Section~3.1, Section~4.2}, one can see that the action of $G$
	on $Y$ is free. Given that $G$ always acts properly discontinuously on $Y$,
	the quotient map $Y \to Y/G$ is a covering of the infinite-genus hyperbolic
	surface $Y/G$. Let $X$ be the Riemann surface obtained from $Y/G$ by
	uniformization. Given that $X$ is not compact, any holomorphic vector bundle
	on $X$ is trivial. Let $\omega$ be a constant non-zero section of the
	canonical bundle of $X$. Then $M=(X,\omega)$ is a translation surface. The
	translation structure on $M$ can be pulled-back to $Y$ via the covering
	$Y\to Y/G$ to obtain a translation surface $\widetilde{M}$. Note that by
	construction $G$ acts on $\widetilde{M}$ by translations (as deck
	transformations), so that $G<\Isom(\widetilde{M})$. To force these groups to
	be equal, we pick a fundamental domain $N\subset\widetilde{M}$ for the
	action of $G$ by translations and mark in~$N$ a polygon $P$ similar to the
	situation when we construct the vertex surface $M_{\rm Id}$ in the
	constructive proof of \cref{thm:realization-isometries-selfsimilar}. Use the
	elements of $G$ to mark a polygon $gP$ on each copy $gN$, $g\in G,$ of $N$.
	Remove the interiors of all polygons $gP$ and identify opposite sides as
	before. The arguments in \cref{lemma:isometries-are-correct} also apply in
	this context and we can thus conclude that $G=\Isom(M)$.  Note that the
	argumentation presented in this paragraph does not apply to deduce results
	for Veech groups since in general $\Aff(M)$ does not act by isometries on
	$M$.
\end{remark}

\medskip

Let us now comment on why it is not trivial (at least to us) to drop the assumption of positive genus in the statement of
\cref{thm:realization-isometries-selfsimilar}. In \cref{fig:GCExample}, we depict four copies of the Euclidean plane and infinite slits on each copy which are labeled with capital letters from $A$ to $L$. Let $M$ be the translation surface obtained by gluing these four planes along infinite slits with the same label. Note that the genus of $M$ is different from zero since we can find non-separating curves such as the blue curve $\alpha$ illustrated in the figure. In general, such a curve may arise due to a non-trivial relation
in the group $G$. Hence, the process of attaching various instances of the
same (genus-zero) vertex surface along infinite slits can result in the formation of a
non-separating curve and hence of genus. This indicates that for arbitrary groups, the construction used in the proof of \cref{thm:realization-isometries-selfsimilar} does not guarantee at each of its steps that one produces a genus-zero surface.

\begin{figure}[tbh]
	\def\svgwidth{.9\textwidth}
	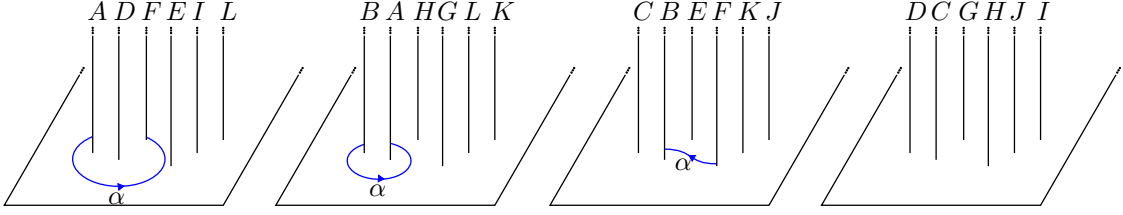
	\caption{A non-separating curve.}
	\label{fig:GCExample}
\end{figure}

However, \cref{thm:realization-isometries-selfsimilar} can be
extended for genus-zero surfaces provided that the group $G$ is~free.

\begin{proposition}\label{prop:genus-zeros-free-group}
	Let $S$ be an infinite-type surface of genus zero, where the end space is
	self-similar. Then, for every $m\geq 1$, there exists a translation structure
	$M$ on $S$ such that $\Isom(M)$ is a free group on $m$ generators.
\end{proposition}

\begin{proof}
	
	The proof of this proposition involves a minor modification of the construction
	used in the proof of \cref{thm:realization-isometries-selfsimilar}. Here,
	we provide the details of the modifications required to obtain the desired
	result.

	\subsubsection*{Blueprint} Let $\mathcal{S} =\{g_1,\dotsc,g_m\}$ be a set
	of generators for $\mathbb{F}_m$ (in particular, without inverses in
	$\mathcal{S}$), and $C_G$ be the corresponding Cayley graph. The vertices of
	$C_G$ are elements of $G$, and for every $g\in G$ and $g_i\in \mathcal{S}$,
	$(g,g_i)$ is a directed edge from $g$ to $gg_i$ (with label $g_i$).
	
	\subsubsection*{Vertex surface $M_{\Id}$} We use a modified
	version of the end-grafting construction to define $M_{\Id}$. Recall that $E(S)\setminus \{x_\infty\}=\sqcup_{n\in\mathbb{N}}E_n$ and let $M'_{\Id}$ be
	the translation surface with end space homeomorphic to~$E_1 \sqcup \{x_\infty\}$ obtained by the
	end-grafting construction. By \cref{prop:MR-construction}, there is an embedded
	copy of a closed half-plane $H_0\hookrightarrow M'_{\Id}$.
	For each $g_n\in \mathcal{S}$, construct two vertical infinite slits in $H_0$
	from~$(4n,0)$ to $x_{\infty,\Id}$ and from~$(4n+1,0)$ to $x_{\infty,\Id}$,
	called $s_{g_n,\text{in}}$ and $s_{g_n,\text{out}}$. We add an infinite slit
	from $(0,0)$ upwards and glue in a finite degree covering of a plane that is
	ramified only over the origin and such that we obtain a unique singularity of total angle $6\pi$ at $(0,0)$. We denote the resulting surface by $M_{\Id}$.
	
	\subsubsection*{Gluings}
	For every $g\in G$, we let $M_g$ denote a copy of $M_{\Id}$, and glue them
	together using the same rules as defined in the proof of
	\cref{thm:realization-isometries-selfsimilar}.
	
	The proof that $M$ is
	homeomorphic to $S$ is similar to the proof of \cref{lemma:topology-is-correct},
	as we do not create genus because the group $G$ is free. As before, there is a
	natural embedding $G\hookrightarrow\Isom(M)$. This embedding is surjective.
	Indeed, observe first that $\{(0,0)\in H_0\subset M_g\}_{g\in G}$ is the set of
	singularities of total angle~$6\pi$ of~$M$, hence invariant by $\Isom(M)$. This
	implies that the set of saddle connections joining $(0,0)$ to $(4,0)$ in~$H_0$
	on each vertex surface $M_g$ have to be permuted by any element of $\Isom(M)$.
	This implies that any element of $\Isom(M)$ permutes the set of saddle connections joining $(0,0)$ to $(4,0)$ in $H_0$ on each vertex surface $M_g$. With similar arguments as in the proof of \cref{lemma:isometries-are-correct}, this implies that such an isometry has to come from an element in $G$.
	To conclude, we apply the arguments in the proof of
	\cref{lemma:isometries-are-correct}.
\end{proof}

\subsection{Proof of \cref{thm:realization-finite-groups}}
\label{ssec:proof-realization-finite-groups}

As in the proof of
\cref{thm:realization-isometries-selfsimilar}, we present a detailed
constructive proof which will be later extended to obtain our main results
about Veech groups. We remark that the argument given after the
constructive proof of \cref{thm:realization-isometries-selfsimilar} also
applies to deduce \cref{thm:realization-finite-groups}
from~\cite{APV}*{Theorem~6.2}.

We construct the translation surface $M$ following~\cite{APV}*{Section~3}, to
which we refer the reader for details. In the following, we describe the
generalities of the construction, emphasizing the differences from the
hyperbolic case.

\subsubsection*{Blueprint} Let $G$ be a finite group and  denote the complete
Cayley graph of $G$ by $C_G$.

\subsubsection*{Vertex surface $M_{\Id}$} Let $E=E(S)=E^g(S)$. Choose an
embedding $E\hookrightarrow\mathbb{S}^2$ and let $\mathbb{S}^2_E$ be the
complement of the image. Let $\{c_i\}_{i\in I}$ be a basis for
$H_1(\mathbb{S}^2_E,\Q)$ formed by pairwise disjoint curves. For each $i\in I$,
let $A_i$ be a regular neighborhood of $c_i$ such that $A_i\cap A_j=\emptyset$
for every $i\neq j$ in $I$ and define $S'$ as the complement of $\sqcup_{i\in
	I}A_i$ in $\mathbb{S}^2_E$. Each $A_i$ is homeomorphic to an open annulus. Let
$S'=\sqcup_{j\in J}S_j$ be the decomposition into connected components.

\begin{remark}\label{rmk:ends-decomposition}
	Each $S_j$ is a planar one-ended surface with $\alpha_j\in\N\cup\{\infty\}$
	boundary components. The union $\sqcup_{j\in J}E(S_j)$ forms a dense subset of
	$E(\mathbb{S}^2_E)$. Indeed, if there was an end not in the closure of $\sqcup_{j\in J}E(S_j)$,
	one could find a class in $H_1(\mathbb{S}^2_E,\Q)$ which cannot be expressed in terms of the base $\{c_i\}_{i\in I}$.
\end{remark}

We define a translation structure $M_{\Id}$ on $\mathbb{S}_E^2$
with additional slits for attachment according to our blueprint. This is done by
defining compatible translation structures with slits on each of the $S_j$ as
follows. For each~$j\in J$, consider a copy $P_j$ of $\mathbb{R}^2$ with its
standard Euclidean metric and a coordinate system $(x,y)$. Denote by
$s_j(x,y)\subset P_j$ the horizontal slit of length $\frac{1}{2}$ whose left
endpoint is $(x,y)$. Choose an enumeration $G=\{g_1,\dotsc,g_N\}$. We mark the
following families of slits in $P_j$:
\begin{itemize}
	\item $s_j(m,0)$, where $1\leq m\leq\alpha_j$.
	\item $s_j(l,k)$, where $1\leq l\leq N$ and $k\in\Z_{>0}$.
\end{itemize}
We continue to denote by $P_j$ the copy of the plane marked along the families
of slits defined above. For every~$j\in J$, let $\partial
S_j=\gamma_{j_i}\sqcup\dotsb\sqcup\gamma_{j_{\alpha_j}}$. Then, for every $i\in
I$, we define gluings in the family $\sqcup_{j\in J} P_j$ as follows: for every
$i\in I$, we have that $\partial A_i=\gamma_{j_m}\sqcup\gamma_{j'_{m'}}$ for
some $j,j'\in J$, $1\leq m\leq \alpha_j$ and $1\leq m'\leq \alpha_{j'}$. Then
glue~$P_j$ to $P_{j'}$ along the slits $s_j(m,0)$ and $s_{j'}(m',0)$. The result
of these gluings is a translation structure~$M_{\Id}'$ on~$\mathbb{S}_E^2$
having only conical singularities of angle $4\pi$ (with additional slits).
Indeed, we can identify the annulus $A_i$ with a small regular neighborhood of
the identified slits $s_j(m,0)=s_{j'}(m',0)$. Then the gluings we just defined
are the same, topologically, as gluing back all annuli in $\{A_i\}_{i\in I}$ to
$S'$.

To define the translation surface $M_{\Id}$ we follow a similar process as in
the proof of \cref{thm:realization-isometries-selfsimilar}. We begin by removing
the interior of a highly asymmetric polygon $P\subset M_{\Id}'$ which does not
intersect any of the slits defined before. Next, identify the opposite sides of
the polygon, which will create a single conical singularity in $M_{\Id}'$ whose
total angle is not $4\pi$. We denote the resulting translation surface as
$M_{\Id}$. It is important to note that, by construction, the group of isometries
$\Isom(M_{\Id})$ is trivial.

\subsubsection*{Gluings} For every vertex $g_l$ of $C_G$, consider a copy
$M_{g_l}$ of $M_{\Id}$. For each edge $(g_l,g_{l'})$ of $C_G$ from $g_l$ to~$g_l
g_{l'}$, we glue $M_{g_l}$ and $M_{g_l g_{l'}}$ as follows. For every
$g_l,g_{l'}\in G$, $j\in J$, and $m\in\Z_{>0}$, glue $s_j(l',2m)$ in $M_{g_l}$ to
$s_j(l',2m-1)$ in $M_{g_lg_{l'}}$. We denote the resulting translation surface
by $M$.

\begin{remark}
	\label{Remark:difference-hypcase}
	In the hyperbolic case, as detailed in~\cite{APV}*{Section~3}, edge surfaces
	are defined using tori with two boundary components. Furthermore, the boundaries
	of these tori exhibit varying lengths to represent the directed edges of the
	complete Cayley graph in the construction of the surface. However, in our
	construction, such tori are not necessary.
	
	In fact, the gluings defined above produce infinite genus (since we glue along
	infinite families of slits) that accumulates to each end of $M_{\Id}$ that
	corresponds to an end in $\sqcup_{j\in J}E(S_j)$. By
	\cref{rmk:ends-decomposition}, then, genus accumulates to all ends in
	$M_{\Id}$. Additionally, we will show that one can enforce $\Isom(M)$ to
	be $G$ by using the set of singularities of $M$ whose total angle is not $4\pi$,
	as in the proof of \cref{lemma:isometries-are-correct}.
\end{remark}

\begin{lemma}
	$\Isom(M)$ is isomorphic to $G$.
\end{lemma}

\begin{proof}
	The natural action of $G$ by left multiplication on $C_G$ provides an
	embedding of~$G$ into $\Isom(M)$, where every element acts as a translation in any local
	coordinate system. To show that $\Isom(M)$ is a subgroup of~$G$, recall that
	when constructing $M_{\Id}$, we removed a highly asymmetric polygon $P$. As
	a result, in each~$M_{g_l}$, there exists a unique frame composed of two
	non-parallel saddle connections of distinct lengths, both of which have an
	endpoint at the sole conical singularity whose angle is not $4\pi$ in
	$V_{g_l}$. By employing the same arguments as used in the proof of
	\cref{lemma:isometries-are-correct}, we can conclude that $\Isom(M)\cong G$.
\end{proof}

\begin{lemma}
	$M$ is homeomorphic to $S$.
\end{lemma}

\begin{proof}
	We claim that:
	\begin{enumerate}
		\item $M/\Isom(M)$ is homeomorphic to $S$, and
		\item given that $G$ is finite, $\Isom(M)$ acts trivially on the end space
		$E(M)$.
	\end{enumerate}
	Assuming these facts, let us proceed with the proof. The covering $\pi\colon
	M\to S=M/\Isom(M)$ is a proper map (since $\Isom(M)$ is finite). Let
	$\hat{\pi}\colon E(M)\to E(S)$ denote the corresponding continuous
	surjective map between end spaces. The map $\hat{\pi}$ is also injective
	because the natural action of $\Isom(M)$ on $M$ at the level of end spaces
	is trivial. Therefore, $\hat{\pi}$ is a continuous bijection between compact
	Hausdorff spaces and thus a homeomorphism. By
	\cref{thm:classification-surfaces}, this implies that $\pi$ is a
	homeomorphism.
	
	We will now demonstrate the validity of both claims above. The gluing rules
	defined earlier imply that~$M/\Isom(M)$ is obtained by gluing,
	\emph{in $M_{\Id}$}, for every $j\in J$, $1\leq l\leq N$, and $m\in\N$, the slit
	$s_j(l,2m)$ to the slit $s_j(l,2m-1)$. These gluings create infinitely many
	handles that accumulate to $\sqcup_{j\in J} E(S_j)$, a dense subset of
	$E(M_{\Id})=E(\mathbb{S}^2_E)=E(S)$. Consequently, $M/\Isom(M)$ is a
	surface with the same end space and genus as $S$.
	
	The proof that $G\cong \Isom(M)$ acts trivially on the end space $E(M)$ is
	fundamentally the same as the proof of~\cite{APV}*{Lemma~3.5}. The main
	idea is as follows.  Take an exhaustion by compact sets $\{K_n'\}_{n\in\N}$
	of $M_{\Id}$. Given that $G$ is finite, we can choose $\{K_n'\}_{n \in \N}$
	such that $\{K_n=\cup_{g\in G}\,gK_n'\}_{n \in \N}$ is a $G$--invariant
	exhaustion of $M$ where each connected component of $M\setminus K_n$ is
	unbounded. Given that $\{K_n\}_{n\in\N}$ is $G$--invariant, $G$ permutes the
	connected components of $M\setminus K_n$. Moreover, for each such component
	$U$ and $g\in G$, we have that $gU=U$ holds. Indeed, one can construct a path
	joining any two points of $U$ and $gU$ using the fact that
	for every $g = g_l$ with $l \in \{1,\ldots,N\}$, there exists $j\in J$ and $k\in \Z_{>0}$ such that the slit $s_j(l,2k)$ is contained in $U$ and the slit $s_j(l, 2k-1)$ is contained in $gU$.
	Then~$gU=U$, implying that $G$ acts trivially on the end space~$E(M)$.
\end{proof}

The previous two lemmas complete the proof of \cref{thm:realization-finite-groups}.\qed

\begin{question}\label{question:genus0-finitegroup}
	For which topological classes of infinite-type genus zero surfaces can we
	(or cannot) realize any finite group $G$ as $\Isom(M)$ for a some
	translation structure $M$ on $S$?
\end{question}

\begin{remark}
	\label{rmk:genus0-finitegroup}
	\begin{enumerate}
		\item In~\cite{APV}*{Proposition~9.1}, it is shown that for every genus-zero
		surface $S$ with non-abelian fundamental group, if $G$ is the isometry group
		of a complete constant-curvature Riemannian metric on $S$ and $G$ is finite,
		then $G$ is isomorphic to a subgroup of $O(3)$.  This result imposes certain
		limitations on the group $G$. It is important to note that flat metrics tend
		to be incomplete, either due to the presence of conical singularities or
		because ends are situated at finite distances.
		
		\item If $S$ is an infinite-type surface featuring a finite number $n$ of
		isolated ends, then any simple non-abelian group $G$ with a cardinality
		greater than $n!$ cannot be realized as an isometry group of a translation
		surface $M$. The proof of this statement can be found
		in~\cite{APV}*{Proposition~9.2}.
	\end{enumerate}
\end{remark}

\subsection{Proof of \cref{thm:countable-case}}
\label{subsec:proof_countable_case}
In the interest of keeping the exposition short and less repetitive, we present a road map to use the results obtained up to this point to deduce the statements in  \cref{thm:countable-case}.

Item (1) follows from
\cref{thm:realization-isometries-selfsimilar}.
Item (2) follows from \cref{lem:translatable_countable_case} and \cref{thm:realization-isometries-translatable}.
~\cite{APV}*{Lemma~4.1} shows that
if $E(S)$ has Cantor--Bendixson degree at least $3$, then $S$ has a finite-type
non-displaceable subsurface, and thus by
\cref{thm:non-displaceable-implies-finite-isometries}, any isometry group must
be finite.
Furthermore, a surface whose end space has characteristic system $(\alpha, 2)$ with $\alpha$ a limit ordinal has doubly-pointed end space but is not translatable by \cref{lem:translatable_countable_case}, hence with \cref{thm:realization-isometries-translatable,thm:realization-isometries-doubly-pointed}, we have again that any isometry group must be finite. By \cref{thm:realization-finite-groups}, Item (3) of the theorem follows.
\qed

\subsection{Proof of \cref{thm:realization-isometries-translatable}}
\label{subsec:proof_realization-isometries-translatable}

\textbf{($\bm{(2) \implies (1)}$):} There exists a translation structure $M$ on $S$ such that
$\Isom(M) = \mathbb{Z}$. Let $h$ be a generator of $\Isom(M)$. We would like to
show that $h$ is a translation in the sense of \cref{def:translatable} (and
therefore $S$ is translatable). For any closed curve $\alpha$ on $S$, consider
the set~$\{h^n(\alpha): n \in \mathbb{Z}\}$. We claim that this set is not
bounded. In fact, if it was bounded, then there existed a compact subsurface~$S'
\subseteq S$ that contains $\{h^n(\alpha) : n \in \mathbb{Z}\}$. The surface
$S'$ then had the property that~$h^n(S') \cap S' \neq \emptyset$ for all~$n \in
\mathbb{Z}$.  This implies that the isometry group does not act properly
discontinuously, which is a contradiction.

Since $h$ is an isometry and hence it acts properly discontinuously, up to replacing $h$ with a power of itself, we may assume that $h(\alpha) \cap
\alpha = \emptyset$. Let $e_+$ and $e_-$ be the two distinguished ends of $S$
with finite orbit under $\Homeo(S)$. With another such replacement, we may
further assume that $h$ fixes each of $e_+$ and $e_-$.

\textbf{Case 1: $\alpha$ separates $e_+$ and $e_-$.} In this case, $h(\alpha)$
must also separate $e_+$ and $e_-$. Since $h(\alpha)$ and~$\alpha$ are disjoint,
we may say without loss of generality that $h(\alpha)$ is contained in the
component of $S \setminus \alpha$ that contains~$e_+$. The surface
$S\setminus\{\alpha\cup h(\alpha)\}$ has three connected components. Let $S'$ be
the connected component whose end space does not intersect $\{e_+,e_-\}$
and $\overline{S'}$ its closure.
Note that $\bigcup_{n \in \mathbb{Z}} h^n(\overline{S'})$ is a connected closed
and open subsurface of~$S$. Since $S$ is connected, this implies that $S' = S$.
Furthermore, we see that $\bigcup_{n > N} h^n(S')$  for any $N$  contains the
end $e_+$, and $\bigcup_{n<N} h^n(S')$ for any $N$ contains the end $e_-$.

Let $U_+$ be a neighbourhood of $e_+$.
If necessary, we may replace it with a smaller neighbourhood that has a
single boundary curve $\gamma_+$. Since $\gamma_+$ is of finite length, there
exists an $N$ such that $\gamma_+$ is contained in $\bigcup_{ -N \leq n \leq N}
h^n(S')$. Then if $n > N$, the subsurface $\bigcup_{n > N} h^n(S')$ is disjoint
from $\gamma_+$ and, as it contains the end $e_+$, the subsurface must be
contained in $U_+$. Finally, this subsurface contains $h^n(\alpha)$ for all $n >
N$. Therefore $\lim_{n \to\infty} h^n(\alpha) = e_+$. We can argue
similarly that $\lim_{n \to -\infty} h^n(\alpha) = e_-$.

\textbf{Case 2: $\alpha$ does not separate $e_+$ and $e_-$.} In this case, there
exists a curve $\beta$ that separates the two ends and we may find a subsurface
$S'$ as above bounded by $\beta$ and $h^m(\beta)$ for some $m \in \mathbb{N}$.
Since $\alpha$ is of finite length, there is an $M \in \mathbb{N}$ such that
$\alpha$ is contained in $S^M \coloneqq \bigcup_{ -M \leq n \leq M}
h^{nm}(S')$.

As in Case 1, $\lim_{n\rightarrow \infty} h^n (\beta) = \lim_{n \rightarrow
	\infty} h^n((h^m)(\beta)) = e_+$. Therefore $\lim_{n \to \infty} h^n(\alpha) =
e_+$, and similarly $\lim_{n \to -\infty} h^n(\alpha) = e_-$. Therefore $S$ is translatable.

\textbf{($\bm{(1) \implies (3)}$):} Let $h$ be a translation on $S$ in the sense of
\cref{def:translatable}. In particular, $h$ has infinite order. We will
construct a translation structure on $S$ such that $h$ is an isometry. By
\cref{lem:translatable_Z-cover}, we have a regular $\Z$--covering $\pi\colon S\to
S'$, where $S'=S/\langle h\rangle$. By
\cref{prop:translation_structure_Riemann_surface}, we can choose a translation
structure~$M'$ on $S'$. The pull-back of $M'$ via the projection map $\pi$
defines a translation structure $M$ on $S$ for which $h\in\Trans(M)<\Isom(M)$.

\textbf{($\bm{(3) \implies (2)}$):} Consider any translation structure on $S$ with isometry
group $G$. Recall that we have assumed that the end space of $S$ is doubly
pointed. Therefore by \cref{thm:realization-isometries-doubly-pointed}, $G$ must
be virtually cyclic. This proves the forward direction of $(2)$.

Now let $G$ be a virtually cyclic group and suppose that there exists a
translation structure on $S$ whose isometry group has an infinite-order
isometry. By \cref{lem:translatable_Z-cover} and
\cref{rmk:translatable_Z-cover}, we have a covering $S\to S'$ whose deck
transformation group is infinite cyclic and the end space of $S$ admits a
decomposition of the form $E(S) = \sqcup_{n \in \mathbb{Z}} E_n \sqcup \{e_+,
e_-\}$ where:
\begin{itemize}
	\item Each $E_n$ is homeomorphic to $E(S')$. Moreover, for every
	$n\neq m$, there exist open subsets $U_n^*,U_m^*$ of $E(S)$ such that
	$E_n\subset U_n^*$, $E_m\subset U_m^*$ and $U_m^*\cap U_n^*=\emptyset$.
	\item Both $e_+$ and $e_-$ are accumulation points of $\sqcup_{n \in
		\mathbb{Z}} E_n$. These are the two distinguished ends of $S$ with finite orbit under $\Homeo(S)$.
\end{itemize}
We use this structure on the space $E(S)$ to construct a translation structure $M$ on $S$ such that $\Isom(M)$ is isomorphic to $G$. There
are two cases to consider.

\textbf{Case 1: $G$ is infinite.}
As in the proof of previous results, we now specify the blueprint, vertex surface, and gluing rules.

\subsubsection*{Blueprint}  Let $\mathcal{S}=\{g_1,\dotsc,g_k\}$ be
a  set of generators for $G$.
Our blueprint will be the directed and labeled Cayley graph $C_G$ of
$G$ with respect to $\mathcal{S}$. Note that  $\Aut(C_G)$ is isomorphic to $G$.

\subsubsection*{Vertex surface $M_{\Id}$} Using the end-grafting construction
and \cref{prop:MR-construction}, we can endow $S'$ with a translation structure
$M_{\Id}'$ where all saddle connections are horizontal, have lengths in $\Z_{>0}$ and all
singularities have total angle $4\pi$. Let $K_0,K_1\subset M_{\Id}'$ be two
compact and disjoint disks that are at distance $10^{1000}$ from any saddle
connection of $M_{\Id}'$, from each other and each has diameter
$\frac{1}{10^{1000}}$ (the choice of these quantities is arbitrary but important
for the rest). For each $g_i\in\mathcal{S}$, we mark horizontal slits of the same length
$s(g_i)^{\pm}$  inside $K_0$ so that no two slits in $\{s(g_i)^{\pm}\}_{i=1}^k$
intersect. As in the proof of~\cref{thm:realization-isometries-selfsimilar}, inside
$K_1$, we carve out a convex $20$--gon $P$ and identify its opposite sides to
create a conical singularity $\sigma_{\Id}$ whose total angle is different from
$4\pi$. Moreover, we can suppose that all sides of $P$ have length
$\frac{l}{10^{1000}}$. We denote the corresponding translation surface by
$M_{\Id}$. We remark that $M_{\Id}$ is not necessarily homeomorphic to $S'$
because the surgery that creates $\sigma_{\Id}$ adds a finite amount of genus.

\subsubsection*{Gluings} For each $g\in G$, let $M_g$ denote a copy of $M_{\Id}$.
For every $(g,g_i)$ in $G\times\{g_1,\dotsc,g_k\}$, we glue the slit~$s(g_i)^-$
in $M_g$ to the slit $s(g_i)^+$ in $M_{gg_i}$. We denote the resulting
translation surface by $M$.

The arguments in the proof of \cref{lemma:isometries-are-correct} apply to $M$, which implies that $G$
is naturally isomorphic to $\Isom(M)$. We now show that $M$ is homeomorphic to $S$.

We claim $g(S)=\infty$. Indeed, if $g(S)>0$ there exist finitely many copies of
$\pi^{-1}(S')$ forming a connected subsurface $S_0\subset S$ of positive genus.
If $h$ denotes the generator of the deck transformation group of the
$\Z$--covering $\pi\colon S\to S'$, then $S=\sqcup_{n\in\Z}h^n(S_0)$ is an
infinite-genus subsurface. We remark that as a consequence~$\{e_+,e_-\}\subset
E^g(S)$.

Let $\widetilde{M}$ denote the translation surface defined by gluing $M_g$ to
$M_{gg_i}$ along a pair of slits as defined above. Since the slits involved in
the gluing are both contained in compact subsets of $M_g$ and $M_{gg_i}$, we have
that~$E(\widetilde{M})$ is homeomorphic to two disjoint copies of $E(S')$. Hence,
$E(M)$ contains a homeomorphic copy of~$\sqcup_{n\in\Z}E_n$, where each $E_n$ is
homeomorphic to $E(S')$ and any two $E_n\neq E_m$ can be separated by an open
subset of $S$. By construction, any element of $E(M)\setminus
\sqcup_{n\in\Z}E_n$ is an end coming from $C_G$. Since $G$ is virtually cyclic,  by~\cite[Satz V]{Hopf44},
$E(G)$ is formed by two points $\{\xi_+, \xi_-\}$. This
implies that $E(M)\setminus \sqcup_{n\in\Z}E_n=\{p_+,p_-\}$. Consider now an
infinite ray $\gamma$ in $C_G$ which converges to $\xi_+$.
Enumerate the vertices in $\gamma$ by~$(v_k)_{k\in\N}$ so that
$\lim_{k\to\infty}v_k=\xi_+$ and let~$E_k$ be the copy of $E(S')$ in
$\sqcup_{n\in\Z}E_n$ corresponding to $v_k$. Then $p_+$ is an accumulation point
of $\sqcup_{k\in\N}E_k$. The same reasoning applies to $p_-$. Hence both $p_+$
and~$p_-$ are accumulation points of $\sqcup_{n\in\Z}E_n$. This shows that~$E(S)$ and $E(M)$ are homeomorphic.

We claim that $E^g(S)\subset E(S)$ and $E^g(M)\subset E(M)$ are homeomorphic.
This follows from two facts. First, $E^g(M_{\Id})=E^g(S')$ since the surgery that
creates $\sigma_{\Id}$ only adds finite genus to $M_{\Id}'$. Second, all genus
created by the gluings can only accumulate to the ends of $C_G$. This follows from the fact that $M_g$ and $M_{gg_i}$ are glued along slits
contained in a compact set that naturally identifies with $K_0$, in particular,
these slits never accumulate to the ends of $S'$. This finishes the proof in
Case 1.

\textbf{Case 2: $G$ is finite.} The proof we present here is slightly
different than in Case 1. As before, we have an infinite cyclic abelian covering $S\to S'$ and a decomposition of $E(S)$ of the form $E(S) = \sqcup_{n \in \mathbb{Z}} E_n \sqcup \{e_+,	
e_-\}$. We construct the desired $M$ as follows. First, for each $n\in\Z$, we construct a
translation surface $M_n$ such that $\Isom(M_n) = G$ and, for every $m\neq n$, $M_m$ is homeomorphic to $M_n$ but not isometric. Each of the surfaces $M_n$ will have an end space which is homeomorphic to $|G|$ copies of $E(S')$. Then we glue $\{M_n\}_{n\in\Z}$  along the Cayley graph of $\Z$ (with respect to the generating set $S=\{1\}$) to obtain $M$.

\subsubsection*{Construction of $M_n$} As in the proof of previous results, we now
specify the blueprint, vertex surface, and gluing rules.

\subsubsection*{Blueprint} As in the case when $G$ was infinite, let $\mathcal{S}=\{g_1,\ldots,g_k\}$ be a set of generators.
Our blueprint is the directed and labeled Cayley graph $C_G$ of $G$ with respect to $S$. Recall that $\Aut(C_G)$ is isomorphic to $G$.

\subsubsection*{Vertex surface} For each $n\in\Z$, we consider a translation structure $M_{\Id,n}$ on $S'$  as in the preceding case, except for the choice of the compact set $K_1$ and the surgeries we performed on it, which
now depend on~$n\in\Z$. More precisely:
\begin{enumerate}
	\item We require $\diam(K_1) =	\frac{1}{10^{|n|+1}}$, so that as $|n|\to\infty$, we have $\diam(K_1)\to 0$. Inside
	$K_1$, we carve out a convex polygon $P$ whose sides can be identified to create
	a unique conical singularity $\sigma_{\Id,n}$ of total angle $2\pi c_n\neq
	4\pi$, for some integer $c_n \ge 3$. The choice is made so that $ c_n\neq 
	c_m$ if $m\neq n$.
	\item For each $g_i\in\mathcal{S}$, we mark horizontal slits $s(g_i)^\pm$ inside $K_0$, so that no two slits intersect. Moreover, for each $i=1,\ldots,n$, we require $|s(g_i)^-|<|s(g_i)^+|$ and that the lengths of all slits $s(g_i)^\pm$ are different. We use the pair of real numbers $|s(g_i)^-|<|s(g_i)^+|$ to label and direct a graph that will allow us to show that $\Isom(M_n)$ is isomorphic to $G$.
\end{enumerate}

\subsubsection*{Gluings}   For each $g\in G$, let $M_{g,n}$ denote a copy of
$M_{\Id,n}$. We denote the copy of the singularity $\sigma_{\Id,n}$ in~$M_{g,n}$ by $\sigma_{g,n}$. We glue the surfaces $\{M_{g,n}\}_{g\in G}$ as follows. For every $(g,g_i)\in G\times\{g_1,\ldots,g_k\}$, let $\mathbb{T}_{g_i}$ be a flat torus on which one can mark two disjoint  slits $s'(g_i)^\pm$, each parallel and of the same length as~$s(g_i)^\pm$, respectively. We then glue the slit $s(g_i)^-$ in $M_{g,n}$ to the slit $s'(g_i)^-$ in $\mathbb{T}_{g_i}$ and the slit $s'(g_i)^+$ in $\mathbb{T}_{g_i}$ to the slit $s(g_i)^+$ in~$M_{gg_i,n}$. We denote the resulting translation surface by $M_n$. Note that by construction $E^g(M_n)\subset E(M_n)$ is homeomorphic to $|G|$ disjoint	copies of $E^g(S')\subset E(S')$.

We claim that $\Isom(M_n)$ is isomorphic to $G$. By construction $G<\Isom(M_n)$. We now argue why $G=\Isom(M_n)$. There is a geometric realization  $C'_G$ of $C_G$ and an embedding $C'_G\hookrightarrow M_n$ such that (i) the vertex $g$ is sent to the singularity $\sigma_{g,n}$ and (ii) the edge in $C'_G$ from $g$ to $gg_i$ is sent to a path $\gamma_{g,gg_i}$ through~$\mathbb{T}_{g_i}$, completely contained in $M_{g,n}\cup \mathbb{T}_{g_i}\cup M_{gg_i,n}$. Since, for all $i=\{1,\ldots,k\}$, we imposed $|s(g_i)^-|<|s(g_i)^+|$ and that the lengths of all slits $s(g_i)^\pm$ are different, the edges of $C'_G$ inherit from the geometry of $M_n$ labels and directions which are in (orientation-preserving) correspondence with those of $C_G$ and that must be respected by the action of $\Isom(M_n)$. Hence  $\Isom(M_n)<\Aut(C_G)=G$ and we conclude our claim.

\subsubsection*{Construction of $M$} Consider the natural identification $\iota_{g,n}: M_{g,n}\to M_{\Id,n}$. Let $s^{\pm}(n)\subset M_{\Id,n}$ be two horizontal slits of length $1$ which are different, modulo identification via $\iota_{g,n}$, from all slits involved in the
construction of $M_n$. In every copy $M_{g,n}\subset M_n$ let $s^{\pm}(g,n) :=\iota_{g,n}^{-1}(s^{\pm}(n))$. For every $(g,n)\in G\times \Z$, we glue~$s^{-}(g,n)\subset M_{g,n}\subset M_n$ to $s^{+}(g,n+1)\subset M_{g,n+1}\subset M_{n+1}$.

By construction, every $g\in G$ defines an isometry $T_{g,n}\in\Isom(M_n)=G$ that sends, for every $h\in G$, the slit~$s^{\pm}(h,n)$ to $s^{\pm}(gh,n)$, respectively. Hence we can glue $(T_{g,n})_{n\in\Z}$ to define an isometry $T_g\in\Isom(M)$ which leaves every $M_n$ invariant and permutes, for fixed $n$, all the copies $\{ M_{g,n}\}_{g\in G}$ in the same way $G=\Aut(C_G)$ permutes the vertices of $C_G$. Hence $G<\Isom(M)$. To see that $\Isom(M)<G$, note that any isometry of $M $ has to leave invariant the set of singularities $(\sigma_{g,n})_{g\in G}$. Hence, by the same arguments as the ones used in the proof of \cref{lemma:isometries-are-correct}, we have that such an isometry has to coincide with some $T_g$.

Recall that $M_n$ is obtained by gluing $|G|$ copies of $S'$ along finitely many tori. Hence $E(M_n)$ is homeomorphic to $|G|$ disjoint copies of $E(S')$. On the other hand, $M$ is obtained by gluing the family $(M_n)$ using as blueprint the Cayley graph of $\Z$ with respect to the generating set $\{+1\}$. Using the same arguments as when $G$ was infinite, we can conclude $E(M)$ is homeomorphic to  $E(M) = \sqcup_{n \in \mathbb{Z}} E_n \sqcup \{e_+,
e_-\}$ and, since $M$ also has infinite genus, it is homeomorphic to $S$.
\qed

\subsection{Proof of~\cref{thm:main}}
\label{ssec:proof-main}
In the interest of keeping the exposition short and less repetitive, we present a road map to use the results obtained up to this point to deduce the statements in~\cref{thm:main}.

Case (1) follows from~\cref{thm:realization-isometries-selfsimilar}.
By~\cref{lem:translatable_non_self-similar_implies_doubly_pointed}, we have in
Case (2) that~$E(S)$ is doubly-pointed. The result in this case follows from
\cref{thm:realization-isometries-doubly-pointed} and
\cref{thm:realization-isometries-translatable}. For Case~(3), remark first that
by~\cref{thm:realization-finite-groups}, any finite group can
be realized as an isometry group. To see that in Case~(3) all isometry groups
have to be finite, note that
by~\cref{lemma:trichotomy-of-surfaces}, we have that
either (3.1) $S$ has a non-displaceable subsurface of finite type or (3.2) $E(S)$ is doubly
pointed but $S$ is not translatable. Case (3.1) follows
from~\cref{thm:non-displaceable-implies-finite-isometries}. For Case (3.2),
use~\cref{thm:realization-isometries-doubly-pointed} to conclude that
$\Isom(M)$ has to be virtually cyclic and then
use~\cref{thm:realization-isometries-translatable} to see that this group
cannot contain an element of infinite order.
\qed

\subsection{Proof of~\cref{thm:improvingAPV1}}
\label{ssec:proof-APV-improved}

By~\cite{APV}*{Theorem~6.2} and~\cref{lemma:trichotomy-of-surfaces}, it is sufficient to show that:
\begin{enumerate}
	\item[(A)] If $S$ is translatable and $E(S)$ is not self-similar then for every
	virtually cyclic group $G$, there exists a complete hyperbolic metric $M$ on $S$ such
	that $\Isom(M)$ is isomorphic to $G$.
	\item[(B)] If $E(S)$ is doubly pointed but $S$ is not translatable, then for any
	complete hyperbolic metric $M$ on~$S$, the group $\Isom(M)$ is finite.
\end{enumerate}

As in the proof of \cref{thm:main}, both (A) and (B) follow from the following
hyperbolic version of \cref{thm:realization-isometries-translatable}:

\begin{lemma} \label{lem:hyperbolic_translatable_virtually_cyclic} Let $S$ be an
	infinite-type surface with doubly pointed end space, $E(S) = E^g(S)$, and $G$ be a group. The
	following are equivalent.
	\begin{enumerate}
		\item $S$ is translatable.
		\item There exists a complete hyperbolic structure $M$ on $S$ with $\Isom(M) \cong G
		\iff G$ is virtually cyclic.
		\item There exists a complete hyperbolic structure $M$ on $S$ such that
		$\Isom(M)$ contains an element of infinite order.
	\end{enumerate}
\end{lemma}

\begin{remark}
	The proof of this lemma closely follows the proof
	of~\cref{thm:realization-isometries-translatable}. Hence, we concentrate our
	argumentation on the aspects that need to be modified in the hyperbolic case.
	Remark that the hypotheses of
	\cref{lem:hyperbolic_translatable_virtually_cyclic} are stronger than those
	in~\cref{thm:realization-isometries-translatable} since we ask $E(S) = E^g(S)$
	rather than just $g(S)>0$. For our purposes, this is enough.
\end{remark}

\begin{proof}
	
	The proof is completed in three parts.
	
	\textbf{($\bm{(2) \implies (1)}$):} Here the proof is verbatim as
	in~\cref{thm:realization-isometries-translatable}.
	
	\textbf{($\bm{(1) \implies (3)}$):} Let $h$ be a topological translation on $S$. As shown in \cref{lem:translatable_Z-cover}, $\langle h \rangle$ acts freely and properly discontinuously. We let
	$S'$ be the quotient of $S$ by $\langle h \rangle$. We note that $S'$ cannot be a sphere because~$S$ is of infinite-type. Furthermore, we can replace $h$ by a sufficiently high power of $h$ such that~$S' = S / \langle h \rangle$ has strictly negative Euler characteristic.
	Therefore it is possible to find a complete hyperbolic metric on $S'$. The pull-back of this
	metric to $S$ is a complete hyperbolic metric for which $h$ is an infinite-order
	isometry.
	
	\textbf{($\bm{(3) \implies (2)}$):} There are two directions to prove in this
	case. The forward direction follows immediately from~\cite{APV}*{Theorem~4.13},
	since $E(S)$ is assumed to be doubly pointed. To prove the converse, we follow
	an argument very similar to the translation surface case, with a modified vertex
	surface and using edge surfaces as in [\emph{Ibid.}].
	
	Let $G$ be a virtually cyclic group. If $G$ is finite, the result follows from
	Theorem~3.7 in [\emph{Ibid.}]. Henceforth we suppose that $G$ is infinite. We
	now specify the blueprint, vertex surface, and gluing rules.

	\subsubsection*{Blueprint} Let $\mathcal{S}=\{g_1^{\pm},\dotsc,g_k^{\pm}\}$ be
	a symmetric set of generators for $G$. Our blueprint will be the Cayley graph $C_G$ of
	$G$ with respect to $S$.
	
	\subsubsection*{Vertex surface} Let $h$ be an infinite-order isometry for some
	complete hyperbolic metric on $S$ (our hypothesis). By
	\cref{lem:translatable_Z-cover}, $S' \coloneqq S/ \langle h \rangle$ is a
	surface. As in the previous part of this proof, we note that $S'$ can be chosen to have negative Euler characteristic. Even more, with the same argument, we can ensure that $S'$ is not a $2$--punctured torus.
	We remove $2k$ disjoint open
	disks from $S'$, one for each element of $\mathcal{S}$. We label the boundary
	components of the resulting surface by these generators and will refer to them
	this way.
	
	Since $S'$ (with the disks removed) has negative Euler characteristic, we can put a complete hyperbolic
	metric $V$ on it so that the corresponding isometry
	group is trivial and such that the $2k$ boundary curves are \emph{short} and of
	pairwise distinct lengths. By ``short'' we mean that the collar lemma can be
	applied and every two curves of these lengths cannot intersect each other.  The
	details of the construction of this metric can be found
	in~\cite{APV}*{Section~3}, after the proof of their Lemma~3.1. By our
	assumptions on $S'$, we can ensure that $V$ has no isometries.
	
	\subsubsection*{Edge surfaces} We construct edge surfaces exactly as
	in~\cite{APV}*{Section~3}. For each $x \in \mathcal{S}$, we construct a complete
	hyperbolic surface $E_x$ as follows. Begin with a torus with two open disks
	removed. Consider a pants decomposition for this surface (consisting of two
	pairs of pants). We choose the interior cuffs of these pants to be of the same
	length as each other, short (as in the preceding paragraph), and of distinct
	length from each of the boundary curves of the vertex surface. We choose the two
	boundary curves of $E_x$ to be isometric to the $x$-- and $x^{-1}$--boundary
	components of the vertex surface, respectively. We will refer to these boundary
	components as the $x$-- and $x^{-1}$--boundary components of $E_x$, respectively. A
	complete hyperbolic metric can be chosen realizing all these lengths.
	
	\subsubsection*{Gluings} We now assemble a complete hyperbolic surface $M$ by
	taking one vertex surface $V_g$ for each vertex~$g$ in the Cayley graph $C_G$
	used as blueprint. Each $V_g$ is isometric to $V$, described above. Given an
	edge of~$C_G$ between $g$ and $gx$, we take~$E_x$ and glue it to $V_g$ by
	identifying the two $x$--boundary components. We glue the resulting surface to
	$V_{gx}$ by identifying the two $x^{-1}$--boundary components.
	
	Note that the hyperbolic metric on the vertex surfaces and the edge surfaces was
	constructed using a pants decomposition of each where all the cuffs were
	metrized as short curves. This creates a pants decomposition~$\mathcal{P}$ of
	$M$ composed entirely of short curves, and $\mathcal{P}$ includes the boundary
	curves of the vertex and edge surfaces.
	
	\begin{lemma}
		$M$ is homeomorphic to $S$.
	\end{lemma}
	
	\begin{proof}
		The proof that $M$ is homeomorphic to $S$ is analogous to the one presented for
		translation surfaces in the proof of
		\cref{thm:realization-isometries-translatable}. We remark first that $2k$ edge
		surfaces are being glued to $V_g$ along the same number of boundary components.
		In particular, all gluings in $V_g$ happen within a compact subset $K_0$.
		Moreover, each vertex surface is a genus--$1$ surface, hence the only ends that
		are accumulated by the genus added from the gluings are those coming from the
		Cayley graph $C_G$. These are only two since $G$ is virtually cyclic and
		infinite.
	\end{proof}
	
	\begin{lemma}
		$\Isom (M)$ is isomorphic to $G$.
	\end{lemma}
	\begin{proof}
		By construction, $G \leq \Isom (M)$. To check the other inclusion, consider
		any $f \in \Isom (M)$. First note that our choice of short curves guarantees
		that the $f$--image of any boundary curve of a vertex surface cannot
		intersect any of the curves in $\mathcal{P}$.  Thus, they must be sent to
		curves of $\mathcal{P}$ itself. Since $f$ is an isometry, we find that $f$
		must in fact permute the boundary curves of the vertex surfaces. Therefore,
		$f$ must permute the connected components of the complement of the boundary
		curves of the vertex surfaces. We have deliberately chosen the vertex
		surfaces not to be homeomorphic to a $2$--punctured torus. Thus, the vertex
		surfaces must be permuted by $f$. Let $g\in G$ be such that $g^{-1}\circ f$
		fixes $V_{\Id}$. Given that $V_{\Id}$ has trivial isometry group, we conclude that $f\in
		G$.
	\end{proof}
	
	The previous two lemmas complete the proof of \cref{lem:hyperbolic_translatable_virtually_cyclic}.
\end{proof}

\section{Proof of main results regarding Veech groups}
\label{sec:proofs-Veech}

The proofs of
\cref{thm:realization-Veech-selfsimilar,thm:realization-finite-Veech-groups}
closely resemble those of
\cref{thm:realization-isometries-selfsimilar,thm:realization-finite-groups}.
However, a key distinction lies in the fact that $G$ must now act on $M$ via
affine homeomorphisms, which are not necessarily isometries. As a consequence,
we need to modify the choice of vertex surfaces $M_g$. Instead of simply using
copies of $M_{\Id}$, we choose surfaces $g\cdot M_{\Id}$, where $\cdot$
represents the natural $\GL^+(2,\mathbb{R})$--action on translation surfaces
through post-composition on charts (we refer to
\cref{sssec:flow-surgeries-action} for a detailed explanation of this action).

Furthermore, to ensure successful gluings, we must also modify how slits are
defined. The specifics of these alterations will be elaborated upon for each
case in the next sections.

\subsection{Proof of \cref{thm:realization-Veech-selfsimilar}}
Recall that $E(S)$ is radially symmetric and can be written as $E(S) = \bigsqcup_{n\in\N}E_n \sqcup \{x_\infty\}$. We perform the following construction:

\subsubsection*{Blueprint} Let $\mathcal{S}$ be a generating set of $G$.
We consider the Cayley graph $C_G$ of $G$ with respect to $\mathcal{S}$. Specifically, the vertices of $C_G$ correspond to the elements of $G$, and the
edges are pairs $(g,gs)\in G\times G$ where $g\in G$ and $s\in\mathcal{S}$.

\subsubsection*{Vertex surface $M_{\Id}$} We adopt a similar approach to
constructing $M_{\Id}$ as in the proof of
\cref{thm:realization-isometries-selfsimilar}. For details, readers are referred
to the mentioned construction. Let $M_{\Id}'$ represent the translation surface
obtained from the end-grafting construction applied to $\overline{E_1} \cap
E^g(S)\subset\overline{E_1}$, and let $H_0$ be an upper Euclidean half-plane in
$M_{\Id}'$ with coordinate system $(x,y)\in\R\times[0,\infty)$.  The element of
$E(M_{\Id}')$ corresponding to the star point $x_\infty\in\overline{E_1}$ is
denoted by $x_{\infty,\Id}$.

Choose an enumeration $\{g_i\}_{i\in I}$ for $\mathcal{S}$. We define several
families of slits in $H_0$ as follows:
\begin{enumerate}
	\item For each $(g_i,n)\in\mathcal{S}\times\N$ let $s(g_i,n)$ be a horizontal
	slit of length $1$ with its left endpoint at $(2i,2(i+n))$.
	\item For each $(g_i,n)\in\mathcal{S}\times\N$ let $s(g_i^{-1},n)$ be a slit in
	$H_0$ with holonomy vector $g_i^{-1}\cdot(1,0)$ such that:
	\begin{itemize}
		\item for every $i\in I$, the family of slits $\{s(g_i^{-1},n)\}_{n\in\N}$
		is contained in an infinite horizontal ``strip'' given by
		$(\R_{>0}\times [a_i,b_i])\cap\{(x,y)\in H_0 : x>y\}$,
		\item the Euclidean distance between any two slits in
		$\{s(g_i^{-1},n)\}_{n\in\N}$ is at least 10,
		\item for every $i\neq j$ in $I$, the infinite strips
		$\R_{>0}\times [a_i,b_i]$ and $\R_{>0}\times [a_j,b_j]$
		do not intersect.
	\end{itemize}
\end{enumerate}

Due to the choices made above, all the slits defined are contained in the
quadrant of $H_0$ where $x, y > 0$. Thus, similar to the proof of
\cref{thm:realization-isometries-selfsimilar}, we can select a convex polygon $P
\subset H_0$ with a diameter less than $1$, positioned a googol away from all
previously defined slits, and having edges that can be identified to create a
conical singularity with a total angle different from $4\pi$. Furthermore, we can assume
that any two sides of~$P$ that are not identified are non-parallel and have
different lengths. Let $M_{\Id}$ denote the result of performing surgery on
$M_{\Id}'$ along $P$. This ensures that $M_{\Id}$ has a unique conical singularity
$\sigma_{\Id}$ whose total angle is not~$4\pi$ and a set of distinguished saddle
connections. Due to the choice of $P$, $\Aff(M_{\Id})$ is trivial.

\subsubsection*{Gluings} For every $h\in G$, define $M_h=h\cdot M_{\Id}$, where $h$
acts on $M_{\Id}$ as a linear transformation fixing the origin $(0,0)\in H_0$. We
use $\sigma_{h}$ and $x_{\infty,h}$ to denote the images of $\sigma_{\Id}$ and
$x_{\infty,\Id}$ in $M_h$ and $E(M_h)$, respectively. Observe that for every
$(h,g_i,n)\in G\times\mathcal{S}\times\N$, the slit $s(g_i,n)$ in~$M_h$ is parallel to and has the same length as
$s(g_i^{-1},n)$ in $M_{hg_i}$.  Consequently, for each $(h,g_i,n)\in
G\times\mathcal{S}\times\N$, we can glue the surfaces $M_h$ and $M_{hg_i}$ along
the slits $s(g_i,n)\subset M_h$ and $s(g_i^{-1},n)\subset
M_{hg_i}$. The resulting translation surface is denoted by~$M$.

According to \cref{rmk:gluing-infinite-slits}, gluing two copies of
$\mathbb{R}^2$ along two infinite families of slits that do not accumulate to a
point produces a Loch Ness monster. By applying this fact each time we glue
$M_h$ and $M_{hg_i}$ for every~$(h,g_i)\in G\times\mathcal{S}$, we deduce that
all elements in $\{x_{\infty,g}\}_{g\in G}$ merge into a single element
$y_\infty\in E(M)$.

\begin{lemma}
	\label{lemma:aff-veech}
	$\Aff(M)$ is isomorphic to $G$.
\end{lemma}
\begin{proof}
	There is a natural embedding of $G$ into $\Aff(M)$: To each $g\in G$, we
	associate $A_g \in\Aff(M)$ such that  the restriction $A_g|\colon M_h\to M_{gh}$
	is, in local coordinates around the origin in $h\cdot H_0\subset M_h$,
	the linear transformation defined by the matrix $g\in\GL^+(2,\mathbb{R})$. Let
	$A\in\Aff(M)$. Then there exists $h\in G$ such that~$A_h^{-1}\circ A
	(\sigma_{\Id})=\sigma_{\Id}$. In particular, $A_h^{-1}\circ A$
	leaves invariant the set of saddle connections joining $\sigma_{\Id}$ to itself.
	If we chose the polygon $P$ as in the proof of
	\cref{thm:realization-isometries-selfsimilar}, we can similarly deduce that
	$A_h=A$ and thus~$\Aff(M)$ is isomorphic to $G$. We remark that, by our
	construction, there are no non-trivial translations in $\Aff(M)$ and so the map
	$\Aff(M) \to \Gamma(M)$ has trivial kernel.
\end{proof}

\begin{lemma}
	$M$ is homeomorphic to $S$.
\end{lemma}
\begin{proof}
	As in \cref{lemma:topology-is-correct}, the proof hinges on proving that we are
	not creating new ends when gluing~$M_h$ to~$M_{hg_i}$ for every $(h,g_i)\in
	G\times\mathcal{S}$. This conclusion follows from
	\cref{rmk:gluing-infinite-slits}.
\end{proof}

\subsection{Proof of \cref{thm:realization-finite-Veech-groups}}

The proof of this theorem closely follows that of
\cref{thm:realization-finite-groups} in \cref{ssec:proof-realization-finite-groups}. We will highlight the modifications that
need to be made and refer the reader to the proof of
\cref{thm:realization-finite-groups} for notation and additional details.

\subsubsection*{Vertex surface $M_{\Id}$}
The construction of the vertex surface is the same, except that this time we
mark slits in $P_j$ in a different way. As in the proof of
\cref{thm:realization-finite-groups}, $s_j(x,y)$ denotes a horizontal slit in
$P_j$ of length $\frac{1}{2}$ (this length is just a convention) whose left
endpoint is $(x,y)$. 
Let $G = \{g_1, \ldots, g_N\}$.
For every $g_l\in G$, we define $s_j(g_l^{-1},x,y)$ as the
slit in $P_j$ with holonomy vector $g_l^{-1}\cdot (\frac{1}{2},0)$ whose left
endpoint is $(x,y)$.

Since $G$ is finite, there exist $x_0,y_0>0$ sufficiently large such that for
every $j\in J$ and $1\leq l\leq N$, all slits in the family
\[
F_l \coloneqq \{s_j(lx_0,2ky_0)\}_{k\in\N}\cup\{s_j(g_l^{-1},lx_0,(2k-1)y_0)\}_{k\in\N}
\]
are disjoint, and for every $l\neq l'$, we have that $F_l\cap F_{l'}=\emptyset$.
We mark these slits and the finite family of slits $\{s_j(m,0)\}$, $1\leq
m\leq\alpha_j$ in $P_j$. Modulo rescaling slits in $\{s_j(m,0)\}$, we can
suppose that this family is disjoint from all slits in $\sqcup_{l=1}^N F_l$. Now
$M_{\Id}$ is defined analogously as follows. Glue the collection of planes~$\sqcup_{j\in J}P_j$ along $\{s_j(m,0)\}$, $j\in J$ and $1\leq m\leq\alpha_j$ as
in the proof of \cref{thm:realization-finite-groups} to obtain a translation
structure $M_{\Id}'$ on $\mathbb{S}_E^2$. Finally, perform a surgery along a
highly asymmetric convex polygon $P$ to create a unique conical singularity of
total angle different from $4\pi$. The resulting translation surface is denoted
by~$M_{\Id}$.

\subsubsection*{Gluings} For each $g\in G$, we define $M_{g}\coloneqq g\cdot
M_{\Id}$, where $g$ acts linearly on each copy of $P_j\subset M_{Id}$. Remark
that by definition of the $\GL^+(2,\mathbb{R})$--action in
\cref{sssec:flow-surgeries-action}, $M_{hg}=hg\cdot M_{\Id}$ because we are considering the action of this group on itself by multiplication on the left. 
Hence for all $g_l,g_{l'}\in G$, each slit $ s_j(lx_0,2ky_0)\subset M_{g_{l'}}$ has the
same holonomy vector as $s_j(s_g^{-1},lx_0,(2k-1)y_0)\subset
M_{g_{l'}g_l}$. We glue then for all $g_l,g_{l'}\in G$ the surfaces~$M_{g_{l'}}$ and $M_{g_{l'}g_l}$ along $s_j(lx_0,2ky_0)\subset M_{g_{l'}}$ and
$s_j(g_l^{-1},lx_0,(2k-1)y_0)\subset M_{g_{l'}g_l}$ for each
$k\in\N$. We denote by $M$ the resulting translation surface.

Note that these gluing rules are topologically and combinatorially
identical to those defined in the proof of
\cref{thm:realization-finite-groups}.
Also, as $G$ is a finite group, its action has the same properties as the action in the proof of \cref{thm:realization-finite-groups}, in particular, it is properly discontinuously.
Hence, $M$ is homeomorphic to $S$.
The proof that~$\Aff(M)\cong G$ is analogous to the proof of
\cref{lemma:aff-veech}.
\qed

\subsection{Proof of \cref{thm:realization-vistually-cyclic-Veech-groups}} The proof we present mimics the proof of \cref{thm:realization-isometries-translatable}. More precisely, we follow the ideas used to show there the implication $(3) \implies (2)$. To avoid repetition of the phrase \emph{as in the the implication $(3) \implies (2)$ of \cref{thm:realization-isometries-translatable}}, we will abbreviate by saying \emph{as in \cref{thm:realization-isometries-translatable}}.

By hypothesis, $S$ is translatable and $E(S)$ is not self-similar, hence by \cref{lem:translatable_non_self-similar_implies_doubly_pointed} we have that $E(S)$ is doubly pointed. Moreover, by \cref{lem:translatable_Z-cover} and \cref{rmk:translatable_Z-cover}, there exists an infinite cyclic covering $S\to S'$ and a decomposition $E(S)=\sqcup_{n\in\Z}E_n\cup\{e_+,e_-\}$ where each $E_n$ is homeomorphic to $E(S')$ and the ends in~$\sqcup_{n\in\Z}E_n$ accumulate to $\{e_+,e_-\}$. We consider two cases. On each case, we specify blueprint, vertex surface, and gluings. 

\textbf{Case 1: G is infinite}. We fix a set of generators $\mathcal{S}$ for
$G$.
The blueprint is
the directed and labeled Cayley graph $C_G$. The vertex surface $M_{\Id}$ is
defined almost identically as in
\cref{thm:realization-isometries-translatable}. The difference is
that now the slits $s(g_i)^\pm$ are neither parallel nor of the same length.
More precisely, we choose for each~$g_i\in\mathcal{S}$ a positive real
number $t_i>0$ so that if~$s(g_i)^-$ is given by a horizontal vector
$t_ie_1$ of length $t_i$ and $s(g_i)^+$ is given by the vector
$t_ig_i^{-1}e_1$ both slits $s(g_i)^\pm$ can be marked inside~$K_0$ and are
disjoint from each other.  Let~$M_g=g\cdot M_{\Id}$. For every $(g,g_i)$ in
$G\times\{g_1,\dotsc,g_k\}$, we glue  the slit~$s(g_i)^-$ in $M_g$ to the
slit $s(g_i)^+$ in~$M_{gg_i}$. These gluings are possible because these
slits are parallel and of the same length. We denote the resulting
translation surface by $M$.

The arguments in the proof of \cref{lemma:aff-veech} apply to $M$, from
where we can conclude that $\Aff(M)$ is isomorphic to $G$. Moreover, $M$ is
homeomorphic to $S$ since the construction is, up to changing the direction
and the sizes of the slits $s(g_i)^\pm$, exactly the same as in
\cref{thm:realization-isometries-translatable}.

\textbf{Case 2: G is finite}. As in the preceding case, we proceed as in the
proof of \cref{thm:realization-isometries-translatable}. The main
difference will be on the choice of slits that we mark to perform the
gluings. We fix a set of generators $\mathcal{S}$ for~$G$.
The blueprint is the directed
and labeled Cayley graph $C_G$. The surface $M_{\Id,n}$ is defined almost
identically as in \cref{thm:realization-isometries-translatable}, the
difference this time being that the slits $s(g_i)^\pm$ are chosen as in the
preceding case in this proof when $G$ was infinite. For each $g\in G$, let
$M_{g,n}:=g\cdot M_{\Id,n}$. Now we make no use of tori~$\mathbb{T}_{g_i}$.
Instead, for every $(g,g_i)$ in $G\times\{g_1,\dotsc,g_k\}$, we simply glue
the slit~$s(g_i)^-$ in $M_{g,n}$ to the slit~$s(g_i)^+$ in~$M_{gg_i,n}$.
These gluings are possible because these slits are of parallel and of the
same length.
We denote the resulting translation surface by $M_n$. The
arguments in the proof of \cref{lemma:aff-veech} apply to $M_n$, from
where we can conclude that $\Aff(M_n)$ is isomorphic to $G$. Note that by
construction $E^g(M_n)\subset E(M_n)$ is homeomorphic to $|G|$ disjoint
copies of $E^g(S')\subset E(S')$. Now to construct the desired surface $M$,
we glue the surfaces $\{M_{g,n}\}_{g\in G, n\in\Z}$ along the Cayley graph
of $\Z$ as follows. Let $s^\pm(\Id)\subset M_{\Id,n}$ be two disjoint
horizontal slits which are disjoint from all slits marked before. We denote
by $s^\pm(g)\subset M_{g,n}$ the image of these by the affine map sending
$M_{\Id,n}$ to $M_{g,n}$. For every $(g,n)\in G\times\Z$, we glue
$s^-(g)\subset M_{g,n}\subset M_n$ to~$s^+(g)\subset M_{g,n+1}\subset
M_{n+1}$. 
These gluings are possible because these slits are parallel and
of the same length. We denote the resulting surface by $M$. The arguments in
the proof of \cref{lemma:aff-veech} apply to $M$, from where we can
conclude that $\Aff(M)$ is isomorphic to $G$. Given that the gluings are,
modulo changing the direction and size of marked slits, the same as in
\cref{thm:realization-isometries-translatable}, we conclude that $M$
is homeomorphic to~$S$. \qed

\subsection{Proof of \cref{thm:uncountable-veech-no-restrictions}} Let $S$ be an
infinite-type surface and $M$ the translation structure on $S$ obtained from the
end-grafting construction applied to $E^g(S)\subset E(S)$. By
\cref{prop:MR-construction}, all saddle connections in~$M$ are horizontal and of integer
length. Therefore $\Gamma(M) < P$, where $P$ is the matrix group defined
in equation~\eqref{eq:group-P} in the statement of the theorem. Now for every $A\in P$,
let $f_{A,n}$ be the linear action of $A$ in the plane $P_n$ used in the
end-grafting construction of $M$ that fixes the line $y=0$ (see
\cref{ssec:Maluendas-Randecker} for details). These affine maps $f_{A,n}$
respect the gluings in the end-grafting construction, and hence can be glued
together to define an element $f_A\in\Aff(M)$ whose derivative is $A$. Hence $P<\Gamma(M)$.
\qed

\section{Statements and declarations}

\textbf{Data availability}. Data sharing is not applicable to this article as no datasets were generated or analysed
during the current study.\\

\textbf{Conflict of interest}. On behalf of all authors, the corresponding author states that there is no conflict of interest.\\

\bibliography{references}{}
\bibliographystyle{plain}

\end{document}